\documentclass[11pt,thmsa]{amsart}

\usepackage{amssymb,latexsym}
\usepackage{amsmath,amscd}
\usepackage[pdftex,all]{xy}
\usepackage{color}
\usepackage{hyperref}
 \usepackage{amsmath}
 \usepackage{amsthm}
 \usepackage{amsfonts}
 \usepackage{amssymb}
 \usepackage[pdftex]{graphicx}
\usepackage{wrapfig}
 \usepackage{layout}
 \usepackage{verbatim}
 \usepackage{alltt}

 \newtheorem{theorem}{Theorem}[section]
 \newtheorem{lemma}[theorem]{Lemma}
 \newtheorem{proposition}[theorem]{Proposition}
 
 \newtheorem{corollary}[theorem]{Corollary}
 \newtheorem{definition}[theorem]{Definition}
  \theoremstyle{definition}
 
 \newtheorem{remark}[theorem]{Remark}

\newtheorem*{acknowledgements}{Acknowledgements}

\newcommand{\g}{\mathfrak{g}}
\newcommand{\id}{\mathrm{id}}
\newcommand{\bs}{\mathbf{s}}

\newcommand{\Q}{\mathbb{Q}}
\newcommand{\R}{\mathbb{R}}

\newcommand{\sM}{\mathcal{M}}
\newcommand{\U}{\mathcal{U}}
\newcommand{\field}{\Bbbk}
\newcommand{\dR}{\mathrm{dR}}

\begin{document}

 \title{How to discretize the differential forms on the interval}

\author{Ruggero Bandiera}
\address{Universit\`a degli studi di Roma La Sapienza,
Dipartimento di Matematica ``Guido Castelnuovo",
P.le Aldo Moro 5, I-00185 Roma, Italy.}
\email{bandiera@mat.uniroma1.it}

\author{Florian Sch\"atz}
\address{University of Luxembourg, Mathematics Research Unit, 
6, rue Coudenhove-Kalergi,
L-1359 Luxembourg, Luxembourg.
}
\email{florian.schaetz@gmail.com}

\maketitle

 \begin{abstract}

We provide explicit quasi-isomorphisms between the following three algebraic structures associated to the unit interval: i) the commutative dg algebra of differential forms, ii) the 
non-commutative dg algebra of simplicial cochains and iii) the Whitney forms, equipped with a homotopy commutative and homotopy associative, i.e. $C_\infty$, algebra structure.
Our main interest lies in a natural `discretization' $C_\infty$ quasi-isomorphism
$\varphi$ from differential forms to Whitney forms.
We establish a uniqueness result that implies that $\varphi$
coincides with the morphism from homotopy transfer,
and obtain several explicit formulas for $\varphi$, all of which are related to the Magnus expansion.
In particular, we recover combinatorial formulas
for the Magnus expansion due to Mielnik and Pleba\'nski.

\end{abstract}

\tableofcontents

\section*{Introduction}

The purpose of this paper is to construct several explicit quasi-isomorphisms between three algebraic structures associated to the unit interval $[0,1]$, and study some of their properties.
The first algebraic structure we consider is the commutative dg algebra of differential forms $\Omega^*([0,1])$ -- with the usual de Rham differential and wedge product.
The other two structures are defined on the subcomplex $C^*([0,1])\subset\Omega^*([0,1])$ of Whitney forms
$$ C^*([0,1]) = \{a t + b (1-t) \, \vert \, a,b\in \field\} \oplus \{c \, dt \, \vert \, c \in \field\},$$
consisting of affine functions and constant one-forms on $[0,1]$. 
Notice that, as a complex, $C^*([0,1])$ is isomorphic to the complex
of simplicial cochains on the one-dimensional simplex: as such, it is equipped with
a dg algebra structure via the usual cup product $\cup$ of cochains, and we denote this dg algebra by
$C^*_\cup([0,1])$. The cup product $\cup$ is {\em not} graded commutative.

On the other hand, since the inclusion $\iota: C^*([0,1]) \hookrightarrow \Omega^*([0,1])$
is a quasi-isomorphism of complexes, the general transfer theorems of homotopical algebra guarantee the existence of a homotopy associative and commutative -- i.e., a $C_\infty$ -- algebra structure on $C^*([0,1])$.
The latter
was worked out explicitly in the papers \cite{Cheng-Getzler,Fiorenza-Manetti,Mnev}, cf. Theorem \ref{theorem: induced structure} in Subsection \ref{subsection: Whitney forms} below, and its Taylor coefficients are given in terms of Bernoulli numbers.
We denote $C^*([0,1])$, equipped with this $C_\infty$ algebra structure,
by $C_\infty^*([0,1])$.

Furthermore, and again by homotopy transfer, one obtains a deformation of the inclusion $\iota:C^*([0,1])\to\Omega^*([0,1])$ into a quasi-isomorphism $\mu:C^*_\infty([0,1])\to\Omega^*([0,1])$ of $C_\infty$ algebras -- see \cite{Cheng-Getzler,Fiorenza-Manetti,Mnev} for explicit formulas in terms of Bernoulli polynomials, and in particular \cite{Cheng-Getzler}  for the verification that one obtains indeed a morphism of $C_\infty$ algebras (a different proof can be found in Appendix \ref{appendix: C-infty}).

The $C_\infty$ algebra $C_\infty^*([0,1])$ 
contains -- in a precise mathematical sense -- the same information
as the commutative dg algebra $\Omega^*([0,1])$.
Since $C_\infty^*([0,1])$ is finite-dimensional, one can think of it as a natural discretization of $\Omega^*([0,1])$ -- see also \cite{Mnev} for the corresponding interpretation of $C_\infty^*([0,1])$ in terms of a discretization of BF-theory on the interval.
The map $\mu$ then provides a canonical way to related the discretization $C_\infty^*([0,1])$ to the original structure on $\Omega^*([0,1])$.

At this point, the following question naturally arises:
\begin{center}
{\em How to explicitly construct a homotopy inverse to $\mu: C^*_\infty([0,1]) \to \Omega^*([0,1])$?}
\end{center}
Or, to put it differently: how to provide a morphism from 
$\Omega^*([0,1])$ to its discretization?
In principle, one can again invoke the general transfer theorems of homotopical algebra, such as those established in \cite{Huebschmann-Kadeishvili,Markl}. However, this turns out to be a non-trivial task -- in particular, we do not know how to obtain explicit formulas
this way.

To circumvent this problem, we make use of the fact that, as a complex, $C^*([0,1])$ coincides with the simplicial cochains on $[0,1]$.
The relation between the dg algebra of smooth, singular cochains $C^*(M)$ on a manifold $M$, and the dg algebra of differential forms $\Omega^*(M)$,
is well-understood:
First, recall that de Rham's Theorem asserts that integration of forms
over simplices
provides a quasi-isomorphism of complexes
$$ \int: \Omega^*(M) \to C^*(M).$$
Moreover, one can prove, cf. \cite{Gugenheim0}, that the chain map $\int$ admits a refinement to an $A_\infty$ quasi-isomorphism
between $(\Omega^*(M),d,\wedge)$ and $(C^*(M),\delta,\cup)$, which implies that integration induces an isomorphism of algebras at the cohomology level. A rather explicit refinement in terms of Chen's iterated integrals was provided by Gugenheim in \cite{Gugenheim}. In our setting, Gugenheim's construction yields an explicit $A_\infty$ quasi-isomorphism
$$ \lambda: \Omega^*([0,1]) \to C^*_\cup([0,1]).$$

We combine $\mu$ and $\lambda$ to produce the following diagram 
$$
\xymatrix{
	&& \Omega^*([0,1]) \ar@/^1.5pc/[dddrr]^{\lambda} 
	\ar@/^1.5pc/[dddll]_(0.35){\varphi}
	&& \\
	&&  && \\
	&&&& \\
	C_\infty^*([0,1]) \ar@/_1.5pc/[rrrr]_{\exp}^\cong
	\ar@/^1.5pc/[uuurr]^{\mu}
	&&&& C_\cup^*([0,1]) \ar@/_1.5pc/[llll]_{\log}^\cong\ar@/^1.5pc/[uuull]^{\gamma}
}
$$
which yields, in particular, an explicit $C_\infty$ morphism $\varphi$
from $\Omega^*([0,1])$ to its discretization $C^*_\infty([0,1])$.
Let us briefly describe the constituencies of the diagram:
\begin{itemize} 
\item[(1)] $\mu$ is the quasi-isomorphism from $C^*_\infty([0,1])$ to $\Omega^*([0,1])$ obtained by homotopy transfer.
\item[(2)] $\lambda$ is a special case of Gugenheim's $A_\infty$ morphism between
differential forms and smooth, singular cochains.
	\item[(3)] $\exp$ is an isomorphism of $A_\infty$ algebras, defined as the composition 
	$$ \exp: C^*_\infty([0,1]) \stackrel{\mu}{\to} \Omega^*([0,1]) \stackrel{\lambda}{\to} C_\cup^*([0,1]),$$
	and $\log=(\exp)^{-1}$ is its inverse.
	Both have a simple description in Taylor coefficients: their linear part is the identity, and the higher Taylor coefficients vanish unless all their arguments are one-cochains, in which case we recover the Taylor coefficients of the functions $\exp(x)-1$ and $\log(x+1)$ respectively.	
	
	\item[(4)] $\gamma$ is an $A_\infty$ morphism right inverse to $\lambda$, defined as the composition 
	$$ \gamma: C^*_\cup([0,1]) \stackrel{\log}{\to} C^*_\infty([0,1]) \stackrel{\mu}{\to} \Omega^*([0,1]).$$
    We derive explicit formulas for $\gamma$ in Proposition \ref{proposition: gamma} in Subsection \ref{subsection: lambda}.
    
    \item[(5)] $\varphi$ is a $C_\infty$ morphism left inverse to $\mu$, defined as the composition 
	$$ \varphi: \Omega^*([0,1]) \stackrel{\lambda}{\to} C^*_\cup([0,1]) \stackrel{\log}{\to} C^*_\infty([0,1]).$$
    The morphism $\varphi$ is the main object of this paper. We shall derive explicit, as well as recursive, formulas for $\varphi$, and find interesting connections with Lie theory and the Magnus expansion.
   \end{itemize}
    Our main results concerning $\varphi$ are:
    \begin{itemize} 
    	\item[(i)] $\varphi$ is indeed a $C_\infty$ morphism, which is not evident from its definition as the composition of two $A_\infty$ morphism. We prove this directly in Corollary \ref{corollary: C-infty} in Subsection \ref{subsection: varphi} and indirectly in Corollary \ref{corollary: uniqueness of gamma}, Subsection \ref{subsection: uniqueness}. For the direct argument, we show the identity $\varphi_n=\lambda_n\circ E^*$, where $\varphi_n,\lambda_n$ are the $n$'th Taylor coefficients of $\varphi$ and $\lambda$ respectively, and $E^*$ is a canonical projector vanishing on the image of the shuffle product. More precisely, $E^*$ is the adjoint to the first Eulerian idempotent $E$, which is a canonical projector from the tensor algebra onto the free Lie algebra, see \cite{Reutenauer}. 
    	\item[(ii)] $\varphi$ (as well as $\lambda$, $\exp$, $\log$) is uniquely characterized by the property that its higher Taylor coefficients vanish whenever one of their arguments is a zero-form. As a consequence, we show that $\varphi$ coincides with the morphism constructed via homotopy transfer formulas, as in \cite{Huebschmann-Kadeishvili,Markl}.
		\item[(iii)] After scalar extension by a dg Lie algebra, our explicit formulas for $\varphi$ recover known formulas for the Magnus expansion, see \cite{Magnus,Iserles-Norsett,Mielnik-Plebanski}.
		
    \end{itemize}  

To add some perspective on the previous diagram, we remark that it continues to make sense after we replace the interval/one-simplex $[0,1]$ by any manifold/simplicial set $M$. We already observed this for Gugenheim's $A_\infty$ morphism $\lambda:\Omega^*(M)\to C^*_\cup(M)$. The $C_\infty$ algebra $C_\infty^*(M)$ and the $C_\infty$ morphism $\mu:C^*_\infty(M)\to\Omega^*(M)$ can be defined as before via homotopy transfer (along Dupont's contraction, see \cite{Dupont2}). Finally, the rest of the diagram can be defined as before: in particular, $\exp:C^*_\infty(M)\xrightarrow{\mu}\Omega^*(M)\xrightarrow{\lambda} C^*_\cup(M)$ continues to be an $A_\infty$ isomorphism with linear part the identity. We remark that the previous diagram is natural in $M$, and in particular our formulas continue to apply when $M$ is a one-dimensional simplicial set. To the authors' knowledge, it is both an hard and interesting open problem to better understand the higher dimensional case. Let us
 point out some topics to which this problem is related:

\begin{itemize}
\item {\em Rational homotopy theory}: the composition of the functor $C^*_\infty(-)$ and the Chevalley-Eilenberg functor from $C_\infty$ algebras to (complete) dg Lie algebras yields a functor $L(-)$ from simplicial sets to (complete) dg Lie algebras, representing the underlying Quillen's equivalence from rational homotopy theory, see \cite{Lie models}. In this context, the $A_\infty$ isomorphism $\exp$ from the previous diagram corresponds to an isomorphism of dg algebras $\Omega C_*(M)\xrightarrow{\cong}\mathcal{U}(L(M))$, where $\mathcal{U}(L(M))$ is the universal enveloping of $L(M)$ and $\Omega C_*(M)$ is the natural simplicial analog of the Adams-Hilton model studied in \cite{Adams-Hilton,Majewski}. More concretely, $\Omega C_*(M)$ is the cobar construction of the dg coalgebra $C_*(M)$ of normalized chains on $M$. It would be interesting to compare the cocommutative dg Hopf algebra structure induced on $\Omega C_*(M)$ by the previous isomorphism and the one studied in the papers \cite{Adams-Hilton}, \cite[App. D]{Majewski}, which is cocommutative only up to homotopy. This would open up the possibility to use the results of the latter reference to get explicit comparisons between $L(M)$ and other classical models for the rational homotopy type of $M$. We briefly address the particular case of $M=[0,1]$ in Remark \ref{rem:AW diagonal}, Subsection \ref{subsection: uniqueness}. In this case, the dg Lie algebra $L([0,1])$ recovers the well-studied Lawrence-Sullivan model of the interval \cite{Lawrence-Sullivan} (as was proved in \cite{Cheng-Getzler}, thus answering a question posed by Sullivan).

\item {\em Derived deformation theory}: the functor $L(-)$ from the previous paragraph is a left adjoint to Getzler's higher generalization of the Deligne groupoid functor, see \cite{Getzler,Getzler-Ham} and the first author's PhD Thesis. In this context, the previous diagram encodes the equivalences between three models of the derived deformation theory associated to a dg Lie algebra $\g$: the one considered by Hinich in \cite{hinichdgC}, the one considered by  Getzler in \cite{Getzler} and the one considered by Behrend and Getzler in \cite[Section 8]{Beherend-Getzler} (the latter makes sense only for dg associative algebras, so either we assume that the Lie bracket on $\g$ is the commutator of an associative product or we replace $\g$ by its universal enveloping algebra). In the one-dimensional case, the three $L_\infty$ algebras $\Omega^*([0,1];\g)$, $C^*_\infty([0,1];\g)$ and $C^*_\cup([0,1];\g)$ obtained via scalar extension by $\g$ (again, the latter makes sense only in the associative setting) encode, via the respective Maurer-Cartan equations (cf. for instance \cite[Section 7]{Fiorenza-Manetti}), three different notions of gauge/homotopy equivalence between Maurer-Cartan elements in the dg Lie/associative algebra $\g$. As is well-known, these three equivalence relations coincide, and our diagram established this fact by providing direct comparisons.

\item {\em Mathematical physics}:  
Let $M$ be an oriented manifold and $\g$ a Lie algebra. From
these data one obtains
a topological field theory on $M$, known as BF-theory.
Its classical action functional reads
\begin{eqnarray*}
 S_{\mathrm{BF}}: \Omega^1(M;\g)\oplus \Omega^{n-2}_{c}(M;\g^*) &\to& \mathbb{R},\\
(A,B) &\mapsto &  \int_M <B,dA + \frac{1}{2}[A,A]> = \int_M <B,F_A>,
\end{eqnarray*}
where $<\cdot,\cdot>: \g^*\times \g \to \mathbb{R}$ is the natural pairing.
In this theory, the induced $L_\infty$ algebra structure on Whitney forms with values in $\g$
corresponds to the tree-level effective action functional $S_{\textrm{eff}}^{\textrm{tree}}$ on the space of infrared fields, obtained by integrating out ultraviolet fields, see \cite{Mnev}.
Moreover, 
the Wilson loop observable $W_\gamma$, given by
$$ W_\gamma(A,B) = \mathrm{tr}(\mathrm{hol}_\gamma(A)),$$
where $\gamma: S^1 \to M$ is a loop and $\mathrm{hol}_\gamma(A)$ is the holonomy
of the connection $A$ around $\gamma$, can be expressed in terms of Chen's iterated integrals. For the case of $[0,1]$, one is therefore naturally led to consider $\lambda$ and $\log \circ \lambda$. We remark that several higher dimensional generalizations
of the Wilson loop observables were constructed and studied in the mathematical physics literature,
see for instance \cite{CR,Pavel-obs}.
\end{itemize}

Let us conclude the introduction of this paper with a brief outline of its structure.

In Section \ref{section: background}, we recall the $C_\infty$ algebra structure on the space of Whitney forms $C^*([0,1])$, along with the $C_\infty$ morphism $\mu$ from $C^*([0,1])$ to $\Omega^*([0,1])$, and Gugenheim's morphism $\lambda$ from differential forms to simplicial cochains. In Subsection \ref{subsection: comparison}, we compute $\exp:=\lambda \circ \mu: C^*_\infty([0,1]) \to C^*_\cup([0,1])$, as well as its inverse $\log$.
Moreover, we work out the morphism
$$ \gamma:C^*_\cup([0,1]) \stackrel{\log}{\to} C^*_\infty([0,1]) \stackrel{\mu}{\to} \Omega^*([0,1])$$
in Subsection \ref{subsection: one-sided}.

 In Section \ref{section: varphi}, we introduce and study the morphism
 $$ \varphi: \Omega^*([0,1]) \stackrel{\lambda}{\to} C_\cup^*([0,1]) \stackrel{\log}{\to} C_\infty^*([0,1]).$$
 We start in \ref{subsection: varphi} by establishing
 explicit formulas for $\varphi$.
The first formula expresses the $n$'th Taylor coefficient $\varphi_n$ of $\varphi$ in terms of an integral over the geometric $n$-simplex, see Theorem \ref{theorem: Eulerian}. In Proposition \ref{theorem: C_oo} we express the Taylor coefficients of $\varphi$ in terms of the adjoint $E^*$ to the first Eulerian idempotent. Together with a symmetry property of $E^*$, Proposition \ref{theorem: C_oo} implies that $\varphi$ is a morphism of $C_\infty$ algebras, see Corollary \ref{corollary: C-infty}. 
In Theorem \ref{theorem: recursion}, we establish a recursive description of $\varphi$, which is inspired by \cite{Iserles-Norsett}.
In Subsection \ref{subsection: uniqueness}, we establish a uniqueness result for morphisms between (very) special $A_\infty$ algebras, and $C_\infty$ algebras, respectively.
This result applies to $\varphi$, and as consequences we deduce that 1) $\varphi$ coincides with the morphism obtained from homotopy transfer and 2) we obtain a second proof that it is a morphism of $C_\infty$ algebras.

In Section \ref{section: pushforward}, we study the pushforward along $\varphi$, after extension of scalars to a dg algebra $A$, or a dg Lie algebra $\g$, respectively. In the latter case, we recover
known formulas for the Mangus expansion.

The two appendices provide background material on $A_\infty$, $L_\infty$ and $C_\infty$ algebras.

\begin{acknowledgements}
We are grateful to Fabian Burghart, Pavel Mn\"ev and
Jan Steinebrunner for generously sharing a draft of their joint work \cite{BMS}, in which they independently derive Proposition  \ref{proposition: ODE1} in Subsection \ref{subsection: coefficients dg Lie algebra} below. Moreover,
F.S. thanks Pavel Mn\"ev for interesting discussions related to the topics of this paper, and
R.B. thanks the University of Luxembourg for hospitality during his visit in December 2015. Finally, we are in debt to James D. Stasheff for numerous valuable comments on a draft version of this paper. 

\end{acknowledgements}

\section{Differential forms on the interval}\label{section: background}

We briefly review three algebraic structures
associated to the interval $[0,1]$, as well as the known morphisms between them.
We refer the reader to the Appendix
for an explanation of our notation and terminology concerning
$A_\infty$, $L_\infty$ and $C_\infty$ algebras.

\subsection{Differential forms and Whitney forms}\label{subsection: Whitney forms}
Throughout the article, $\Omega^*([0,1])$ denotes the graded vector space
of differential forms on the closed interval $[0,1]$.
To be more precise, there are two variants of $\Omega^*([0,1])$ which we will consider:
\begin{itemize}
\item the space of real-valued, smooth differential forms on $[0,1]$, denoted by $\Omega_{\dR}^*([0,1])$, 
equipped with the structure of a commutative dg algebra via the de Rham differential $d$ and the wedge product $\wedge$.

\item the space of $\field$-polynomial forms on $[0,1]$, where $\field$ is a field of characteristic zero, denoted by $\Omega_\field^*([0,1])$: formally, $\Omega_\field^*([0,1])=\Omega_\field^0([0,1])\oplus\Omega_\field^1([0,1])=\field[t]\oplus\field[t]dt$, where $\field[t]$ is the polynomial algebra over $\field$. Again, this is a commutative dg algebra via the wedge product $p(t)\wedge q(t)=p(t)q(t)$, $p(t)\wedge q(t)dt=p(t)q(t)dt$, $p(t)dt\wedge q(t)dt=0$ and the differential $d:p(t)\mapsto p'(t)dt$ and $d:p(t)dt \mapsto 0$.
\end{itemize}
Since most of our constructions work in both contexts, we will usually
just use $\Omega^*([0,1])$ to refer to either variant.

The subcomplex of Whitney forms is the graded vector subspace of $\Omega^*([0,1])$
given by the affine functions and constant one-forms, i.e.
$$ C^*([0,1]) = C^0([0,1])\oplus C^1([0,1])=\{a t + b (1-t) \, \vert \, a,b\in \field\} \oplus \{c \, dt \, \vert \, c \in \field\}.$$

Notice that this space is closed under the differential $d$, but not under multiplication.
However, $C^*([0,1])$ can be identified with the complex of simplicial $\field$-valued cochains on the standard $1$-dimensional simplex. As such, we might equip $C^*([0,1])$ with the cup product $\cup$, which is determined by the fact that the constant function $1$ is a unit and that the relations
$$ t\cup t = t, \quad t \cup dt = 0 \quad \textrm{and} \quad dt \cup t = dt$$
hold.
The cup product is associative and compatible with $d$, hence it makes $C^*([0,1])$
into a dg algebra. We denote this dg algebra by $C^*_{\cup}([0,1])$,
and emphasize that the cup product is not graded commutative.

In order to retain some form of commutativity on $C^*([0,1])$, one can use homological perturbation theory, as done in the references \cite{Huebschmann-Kadeishvili,Markl}, to transfer the wedge product on $\Omega^*([0,1])$ down to a
homotopy associative and homotopy commutative algebra structure, i.e., a $C_\infty$-algebra
structure, on $C^*([0,1])$. We refer to the Appendix for a short reminder on these algebraic structures. To carry out the transfer
of the wedge product from $\Omega^*([0,1])$ to $C^*([0,1])$, we first need to fix suitable contraction data from $\Omega^*([0,1])$ to $C^*([0,1])$. Following \cite{Cheng-Getzler,Fiorenza-Manetti,Mnev} we consider Dupont's contraction (cf. \cite{Dupont2}), which is given by the inclusion
$ \iota: C^*([0,1]) \hookrightarrow \Omega^*([0,1])$, the chain map
\begin{eqnarray*}
 \pi: \Omega^0([0,1]) \to C^0([0,1]), \quad f &\mapsto& f(1)t + f(0)(1-t)\\
\pi: \Omega^1([0,1]) \to C^1([0,1]), \quad a(t) dt & \mapsto & \left(\int_0^1 a(\tau) d\tau \right) dt
\end{eqnarray*}
and the chain homotopy
\begin{eqnarray*}
h: \Omega^1([0,1]) \to \Omega^0([0,1]), \quad a(t)dt \mapsto t\int_0^1 a(\tau) d\tau - \int_0^t a(\tau)d\tau .
\end{eqnarray*}
We notice that the side-conditions
$$ h \circ h=0, \quad h \circ \iota=0 \quad \textrm{and} \quad \pi \circ h=0$$
are satisfied. 

The resulting homotopy algebra structure on $C^*([0,1])$ was explicitly worked out in \cite{Cheng-Getzler,Fiorenza-Manetti,Mnev}.
Below we denote by $\bs$ the suspension endofunctor on the category of graded vector spaces, see Appendix \ref{appendix: A-infty and L-infty} for our conventions related to $A_\infty$ algebras.

\begin{theorem}[\cite{Cheng-Getzler,Fiorenza-Manetti,Mnev}]\label{theorem: induced structure}
The maps $m_{n+1}:\bs C^*([0,1])^{\otimes n+1}\to\bs C^*([0,1])$, ${n\geq1}$, determined by
\begin{itemize}
\item unitality with respect to the constant function $1$,
\item $m_2(\bs t\otimes\bs t)=\bs t$, $m_2(\bs t\otimes\bs dt) =\frac{1}{2}\bs dt$, $m_2(\bs dt\otimes\bs t)=-\frac{1}{2}\bs dt$,
\item for $n>1$ the map $m_{n+1}$ vanishes unless precisely one of its arguments is a function and one has
$$ m_{n+1}((\bs dt)^{\otimes i}\otimes \bs t \otimes (\bs dt)^{\otimes n-i}) = \big((-1)^{i+1}{n \choose i} \frac{B_n}{n!}\big) \bs dt,$$

\end{itemize}
equip the complex
$(C^*([0,1]),d)$ with the structure of a unital $C_\infty$-algebra.
Here $B_n$ is the $n$'th Bernoulli number, defined in terms of the generating function
 $$ \frac{z}{e^{z}-1}=\sum_{n \ge 0}\frac{z^n}{n!}B_n.$$
\end{theorem}

We denote $C^*([0,1])$, equipped with the $C_\infty$ algebra structure
given by the maps
$(d,m_2,m_3,\dots)$, by $C^*_\infty([0,1])$.

The $C_\infty$ algebra structure on $C^*([0,1])$ comes with a quasi-isomorphism of $C_\infty$ algebras
$$\mu: C_\infty^*([0,1]) \to \Omega^*([0,1]),$$
whose linear part is the inclusion
$C^*([0,1])\hookrightarrow \Omega^*([0,1])$,
see \cite{Cheng-Getzler}. 
Explicit formulas for $\mu$ were worked out
in \cite{Fiorenza-Manetti,Mnev}. 

\begin{proposition}[\cite{Cheng-Getzler,Fiorenza-Manetti,Mnev}]
There is a $C_\infty$ morphism
$$ \mu: C^*_\infty([0,1]) \to \Omega^*([0,1])$$
whose Taylor coefficients are determined as follows:
\begin{itemize}
\item $\mu$ is unital.
\item The linear part $\mu_1$ is the inclusion.
\item For $n\geq1$, $\mu_{n+1}$ vanishes unless precisely one of the inputs is a function and one has
$$
\mu_{n+1}((\bs dt)^{\otimes i}\otimes \bs t \otimes (\bs dt)^{\otimes n-i})=\bs\left((-1)^i{n \choose i}\frac{B_{n+1}(t)-B_{n+1}}{(n+1)!} \right).
$$
\end{itemize}
Here $B_n(t)$ is the $n$'th Bernoulli polynomial, defined in terms of
the generating function
$$ \frac{z e^{tz}}{e^{z}-1} = \sum_{n\ge 1}\frac{z^n}{n!} B_n(t).$$
\end{proposition}

\subsection{Gugenheim's $A_\infty$-morphism $\lambda$}\label{subsection: lambda}

Let $X$ be a smooth manifold.
In \cite{Gugenheim} Gugenheim constructed an $A_\infty$ quasi-isomorphism 
$\lambda_X$ from
the de Rham dg algebra
$\Omega^*_{\dR}(X)$ of smooth, real-valued differential forms on $X$ to the dg algebra 
of singular, smooth $\mathbb{R}$-valued cochains on $X$. The construction relies on Chen's theory
of iterated integrals \cite{Chen1}, see also the exposition in \cite{Abad-Schaetz_integration}.

We obtain the following result when we specialize Gugenheim's construction to $X=[0,1]$:

\begin{theorem}
There is a unital $A_\infty$-morphism
$ \lambda: \Omega^*([0,1]) \to C_{\cup}^*([0,1])$
whose Taylor coefficients are determined as follows:
\begin{itemize}
\item The linear part $\lambda_1$ is the chain map $\pi$ from Subsection \ref{subsection: Whitney forms}.
\item For $n>1$, $\lambda_n$ vanishes on tensor products that contain a factor which is a zero-form.
\item For $n\ge 1$ we have
$$ \lambda_n(\bs a_1(t)dt \otimes \cdots \otimes \bs a_n(t)dt)=
\left(\int\limits_{0\le t_1 \le \cdots \le t_n \le 1} a_1(t_1) \cdots a_n(t_n) dt_1\cdots dt_n\right) \bs dt.$$
\end{itemize}
\end{theorem}
We provide a direct proof of this fact below.
\begin{remark}
We remark that the previous theorem remains true when $\Omega^*([0,1])=\Omega^*_\field([0,1])$ is the dg algebra of $\field$-polynomial forms: in this case, given $p(t_1,\ldots,t_n)\in\field[t_1,\ldots,t_n]$ and $s\in\field$, the integral $\int_{0\leq t_1\leq\cdots\leq t_n\leq s}p(t_1,\ldots,t_n)dt_1\cdots dt_n$ can be evaluated formally by setting
$$\int_{0\leq t_1\leq\cdots\leq t_n\leq s}t_1^{l_1-1}\cdots t_n^{l_n-1}dt_1\cdots dt_n=
\frac{s^{l_1+\cdots+l_n} }{l_1(l_1+l_2)\cdots (l_1+\cdots + l_n)}$$
for all positive intgers $l_1,\ldots,l_n$. 
\end{remark}
\begin{proof}
Let us evaluate the defining relations for $\lambda$ to be an $A_\infty$ morphism
on a tensor product of elements in $\bs \Omega^*([0,1])$. We do this by considering three separate cases, which cover all possibilities.

First, suppose all factors are one-forms. Then by degree reasons, the defining relation takes values in the component of degree two of $C^*([0,1])$, which is zero.

The second case to consider is that two or more of the factors are zero-forms. Since
$\lambda_n$ vanishes for $n>1$ if one of the inputs is a zero-form, the defining relation is trivially satisfied in this case as well, unless we consider
precisely $\bs f_1(t) \otimes \bs f_2(t)$. Then the defining relation for $\lambda$ to be an $A_\infty$ morphism reads
$$\pi(f_1(t)f_2(t)) = \pi(f_1(t))\cup  \pi(f_2(t)),$$
which follows immediately from the definitions of $\pi$ and $\cup$.

Finally, we consider an element of the form
$$\bs a_1(t)dt \otimes \cdots \otimes \bs a_i(t)dt \otimes \bs f(t) \otimes \bs a_{i+1}(t)dt \otimes \cdots \otimes \bs a_n(t) dt$$
with $n>0$
and work out the defining relations of $\lambda$ being an $A_\infty$ morphism, evaluated on such an element.

If $0<i<n$ we obtain
\begin{eqnarray*}
&& \int\limits_{{0\le t_1 \le \cdots \le t_{n+1} \le 1}} a_1(t_1) \cdots a_{i}(t_i) \left(\frac{d f}{dt}(t_{i+1})\right) a_{i+1}(t_{i+2})\cdots a_n(t_{n+1}) dt_1\cdots dt_{n+1} \\
&& \stackrel{!}{=} \int\limits_{{0\le t_1 \le \cdots \le t_n \le 1}} a_1(t_1) \cdots a_i(t_i) \Big(f(t_{i+1})a_{i+1}(t_{i+1}) \Big) a_{i+2}(t_{i+2}) \cdots a_n(t_n) dt_1\cdots dt_n \\
&& \,\,\,\, - \int\limits_{{0\le t_1 \le \cdots \le t_n \le 1}} a_1(t_1) \cdots a_{i-1}(t_{i-1}) \Big(a_{i}(t_i)f(t_i)\Big) a_{i+1}(t_{i+1})\cdots a_n(t_n) dt_1\cdots dt_n 
\end{eqnarray*}
which is a consequence of Stokes theorem.
For the extremal case $i=0$, we obtain,
\begin{eqnarray*}
&& \int\limits_{{0\le t_1 \le \cdots \le t_{n+1} \le 1}} \left(\frac{df}{dt}(t_1) \right) a_1(t_2) \cdots a_n(t_{n+1}) dt_1\cdots dt_{n+1}\stackrel{!}{=}\\ 
&& \hspace{1cm}  \int\limits_{{0\le t_1 \le \cdots \le t_n \le 1}} \Big(f(t_1) a_1(t_1)\Big)a_2(t_2) \cdots a_n(t_n) dt_1 \cdots dt_n \\
&& \hspace{1cm} \,\,\,\, - f(0) \left( \int\limits_{{0\le t_1 \le \cdots \le t_n \le 1}} a_1(t_1) \cdots a_{n}(t_n) dt_1\cdots dt_n \right),
\end{eqnarray*}
while for $i=n$, we obtain
\begin{eqnarray*}
&& \int\limits_{{0\le t_1 \le \cdots \le t_{n+1} \le 1}}  a_1(t_1) \cdots a_{n}(t_{n}) \left(\frac{df}{dt}(t_{n+1}) \right)dt_1\cdots dt_{n+1}\stackrel{!}{=}\\ 
&& \hspace{1cm} \left( \int\limits_{{0\le t_1 \le \cdots \le t_n \le 1}} a_1(t_1) \cdots a_n(t_n) dt_1 \cdots dt_k \right) f(1) \\
&& \hspace{1cm} \,\,\,\, - \int\limits_{{0\le t_1 \le \cdots \le t_n \le 1}} a_1(t_1) \cdots a_{k-1}(t_{k-1}) \Big( a_n(t_n) f(t_n) \Big) dt_1\cdots dt_n .
\end{eqnarray*}
Also the latter two equations are immediate consequences of Stokes theorem.
\end{proof}
It is well known that iterated integrals behave well with respect to the shuffle product \cite{Chen1}: in the case of the interval, we have the following proposition, which we will use in the next section. 
\begin{proposition}\label{proposition: guggheneim shuffle} We denote by $p\lambda:\overline{T}(\bs \Omega^1([0,1]))\to \bs C^1([0,1])=\field$ the corestriction of the degree zero part of Gugenheim's morphism: then $p\lambda$ is a morphism of commutative algebras, where we equip $\overline{T}(\bs \Omega^1([0,1]))$ with the shuffle product $\circledast$ (cf. Appendix \ref{appendix: C-infty}).
\end{proposition}
\begin{proof} Denoting by $\Delta_i=\{(t_1,\ldots,t_i)\in\R^i\,\,\mbox{s.t.}\,\, 0\leq t_1\leq\cdots\leq t_i\leq1 \}$ the $i$-dimensional simplex, we have 
	\begin{multline*} \lambda_j\big(\bs a_1(t)dt\otimes \cdots\otimes \bs a_j(t)dt\big)\cdot\lambda_k\big(\bs a_{j+1}(t)dt\otimes\cdots\otimes\bs a_{j+k}(t)dt\big)=\\= \int_{\Delta_j\times\Delta_k}a_1(t_1)\cdots a_{j+k}(t_{j+k})dt_1\cdots dt_{j+k},\end{multline*}
	Using the natural triangulation of $\Delta_j\times\Delta_k$ 
	\begin{eqnarray*}
		\coprod_{\sigma \in S(j,k)} \Delta_n \xrightarrow{\sigma} \Delta_{j}\times\Delta_{k}, \quad
		(t_1,\dots,t_n) &\stackrel{\sigma}{\mapsto}& (t_{\sigma(1)},\dots,t_{\sigma(n)}),
	\end{eqnarray*}
	where $S(j,k)$ is the set of $(j,k)$-unshuffles, we can rewrite the right hand side of the previous equation as 
	\begin{eqnarray*} && \int_{\Delta_j\times\Delta_k}a_1(t_1)\cdots a_{j+k}(t_{j+k})dt_1\cdots dt_{j+k} = \\&&= \sum_{\sigma\in S(j,k)}\int_{\Delta_n}a_{1}(t_{\sigma(1)})\cdots a_{j+k}(t_{\sigma(j+k)})dt_1\cdots dt_{j+k}\\&&= \sum_{\sigma\in S(j,k)}\int_{\Delta_n}a_{\sigma^{-1}(1)}(t_1)\cdots a_{\sigma^{-1}(j+k)}(t_{j+k})dt_1\cdots dt_{j+k}\\&&=\lambda_n\big((\bs a_1(t)dt\otimes \cdots\otimes\bs a_j(t)dt)\circledast(\bs a_{j+1}(t)dt\otimes\cdots\otimes\bs a_{j+k}(t)dt)\big),\end{eqnarray*}
	by definition of the shuffle product $\circledast$.\end{proof}

\subsection{Comparing two structures on Whitney forms}\label{subsection: comparison}

We can now combine the $C_\infty$ morphism $\mu: C^*_\infty([0,1]) \to \Omega^*([0,1])$
with the $A_\infty$ morphism $\lambda: \Omega^*([0,1]) \to C^*_\cup([0,1])$.
Since the linear part of the composition $\lambda \circ \mu$ is the identity,
we obtain an $A_\infty$ isomorphism between $C^*_\infty([0,1])$
and $C^*_\cup([0,1])$.

\begin{proposition}
The Taylor coefficients of the unital $A_\infty$ isomorphism
$$ \exp:=\lambda \circ \mu: C^*_\infty([0,1])\to C^*_\cup([0,1])$$
are determined as follows:
\begin{itemize}
\item The linear part $\exp_1$ is the identity.
\item For $n > 1$, $\exp_n$ vanishes on tensor products that contain a factor of degree $0$.
\item For $n\ge 1$, we have
$$ \exp_n(\bs dt \otimes \cdots \otimes \bs dt) = \frac{1}{n!} \bs dt.$$
\end{itemize}

The inverse $\log: C^*_\cup([0,1]) \to C^*_\infty([0,1])$
to $\exp$ is the unital $A_\infty$ isomorphism whose Taylor coefficients are determined
as follows:
\begin{itemize}
\item The linear part of $\log_1$ is the identity.
\item For $n>1$, $\log_n$ vanishes on tensor products that contain a factor of degree $0$.
\item For $n\ge 1$, we have
$$ \log_n(\bs dt \otimes \cdots \otimes \bs dt) = \frac{(-1)^{n+1}}{n} \bs dt.$$
\end{itemize}
\end{proposition}

\begin{proof}
Let us first consider the map $\exp$: we already observed the assertion about the linear part.

If we evaluate $\exp_{n+1}$, $n\geq1$, on a tensor product of the form
$$(\bs dt)^{\otimes i}\otimes \bs t \otimes (\bs dt)^{\otimes n-i},$$
only the contribution from $\pi\mu_{n+1}((\bs dt)^{\otimes i}\otimes \bs t \otimes (\bs dt)^{\otimes n-i})$ can be non-zero, since all higher order terms 
of $\lambda$ map tensor products which contain a factor that is a zero-form to zero. Hence we obtain
\begin{eqnarray*}
\exp_{n+1}((\bs dt)^{\otimes i}\otimes \bs t \otimes (\bs dt)^{\otimes n-i}) &=&
(-1)^i{n \choose i}\left(\frac{B_{n+1}(1)-B_{n+1}}{(n+1)!}\right) \bs t=0, \\
\end{eqnarray*}
since $B_{n+1}(1)=B_{n+1}$ for $n\geq1$.

On the other hand, only $\lambda_{n}\mu_1^{\otimes n}$ contributes to the evaluation of $\exp_n$ on the tensor product
$ (\bs dt)^{\otimes n}$, since the higher order terms 
of $\mu$ vanish unless precisely one argument is a function, and we find
$$ \exp_n(\bs dt \otimes \cdots \otimes \bs dt) = \left(\int\limits_{0\le t_1\le \cdots \le t_n \le 1} dt_1 \cdots dt_n \right) \bs dt = \frac{1}{n!} \bs dt, $$
as desired. Finally, it is clear by degree reasons that $\exp_n$ vanishes if two or more arguments are functions.

It is straightforward to check that $\log$ as defined in the proposition is indeed the inverse to $\exp$.
\end{proof}

\subsection{A one-sided inverse to $\lambda$}
\label{subsection: one-sided}

We define an $A_\infty$ morphism $\gamma$ as the composition
$$
\xymatrix{
\gamma: C^*_\cup([0,1]) \ar[r]^{\log} & C_\infty^*([0,1]) \ar[r]^\mu & \Omega^*([0,1]).
}
$$
By construction, we have $\lambda \circ \gamma = \lambda \circ \mu \circ \log = \exp \circ \log = \mathrm{id}$.

\begin{proposition}\label{proposition: gamma}
The Taylor coefficients of the unital $A_\infty$ morphism
$$ \gamma = \mu \circ \log: C^*_\cup([0,1]) \to \Omega^*([0,1])$$
are determined as follows:
\begin{itemize}
\item The linear part $\gamma_1$ is the inclusion $C^*([0,1]) \hookrightarrow \Omega^*([0,1])$.
\item For $i\ge 0$, $j\ge 0$, we have
$$ \gamma_{i+j+1}((\bs dt)^{\otimes i}\otimes \bs t \otimes (\bs dt)^{\otimes j}) = \bs \sum_{l=0}^i {1-t \choose l}{t \choose i+j+1-l},$$
where ${\tau \choose n }= \frac{\tau(\tau-1)\cdots (\tau-n+1)}{n!}$.
\item For $n\ge 1$, we have
$$ \gamma_n(\bs dt \otimes \cdots \otimes \bs dt) = \frac{(-1)^{n+1}}{n}\bs dt.$$
\end{itemize}
\end{proposition}

\begin{proof}
We introduce the following generating function
$$ F(z,w):=\sum_{i,j\ge 0}\gamma_{i+j+1}((\bs dt)^{\otimes i}\otimes \bs t \otimes (\bs dt)^{\otimes j})z^i w^j$$
and compute
\begin{eqnarray*}
 F(z,w) &=& \sum_{i,j\ge 0 }\sum_{\stackrel{i_1+\cdots+i_p= i}{j_1+\cdots+ j_q=j}} z^iw^j\frac{(-1)^{i+p}}{i_1\cdots i_p}\frac{(-1)^{j+q}}{j_1\cdots j_q}\mu_{p+q+1}((\bs dt)^{\otimes p}\otimes\bs t\otimes (\bs dt)^{\otimes q}) \\
 &=& \sum_{p,q\ge0}\log(1+z)^p\log(1+w)^q\mu_{p+q+1}((\bs dt)^{\otimes p}\otimes\bs t\otimes (\bs dt)^{\otimes q})  \\
 & =& \sum_{n\ge 0}\left(\sum_{p+q=n}\binom{n}{p}(-1)^p\log(1+z)^p\log(1+w)^q\right)\frac{B_{n+1}(t)-B_{n+1}}{(n+1)!} \\ 
&=& \sum_{n\ge 0}\frac{B_{n+1}(t)-B_{n+1}}{(n+1)!}(\log(1+w)-\log(1+z))^n \\
&=& G\left(\log\left(\frac{1+w}{1+z} \right)\right),
\end{eqnarray*}
where $G(u)$ is the formal power series
$$ G(u)=\sum_{r\ge0} \frac{B_{r+1}(t)-B_{r+1}}{(r+1)!}u^r = \frac{e^{tu}-1}{e^u-1}.$$

Hence we find
$$ F(z,w) = \frac{\left(\frac{1+w}{1+z}\right)^t -1}{\frac{1+w}{1+z}-1}=(1+z)^{1-t}\frac{1}{w-z}\left((1+w)^t - (1+z)^t \right).$$
Since $z$ and $w$ are formal variables, we can apply Newton's generalized binomial Theorem to obtain
$$
\frac{1}{w-z}\left((1+w)^t - (1+z)^t \right)=\frac{1}{w-z}\sum_{k\ge 0}{t \choose k+1}(w^{k+1}-z^{k+1})=
\sum_{r,s\ge0}{t\choose r+s+1} w^r z^s
$$
where, by definition,
$$ {t \choose k}:= \frac{t(t-1)\cdots (t-k+1)}{k!}.$$
We therefore have
$$ F(z,w) = \left(\sum_{l\ge 0}{1-t \choose l}z^l \right)\left(\sum_{r,s\ge0}{t\choose r+s+1} w^r z^s\right).$$
Consequently, the coefficient for $z^i w^j$ of $F(z,w)$
is
$$ \sum_{l=0}^i {1-t \choose l}{t \choose i+j+1-l}.$$

We conclude that
$$\gamma_{i+j+1}((\bs dt)^{\otimes i}\otimes \bs t \otimes (\bs dt)^{\otimes j}) = \bs \sum_{l=0}^i {1-t \choose l}{t \choose i+j+1-l}.$$

Since $\mu_n$ vanishes for $n>1$ if we evaluate it on a tensor product that contains only elements of degree one, the only relevant contribution to
$\gamma_n(\bs dt \otimes \cdots \otimes \bs dt)$ is 
$(\mu_1 \circ \log_n)(\bs dt \otimes \cdots \otimes \bs dt) = \frac{(-1)^{n+1}}{n}\bs dt.$
\end{proof}

\section{The $C_\infty$ morphism $\varphi$ from $\Omega^*([0,1])$ to $C^*_\infty([0,1])$}\label{section: varphi}

In this section we study the composition
$$
\xymatrix{
\varphi: \Omega^*([0,1]) \ar[r]^{\lambda} & C^*_\cup([0,1]) \ar[r]^{\log} & C^*_\infty([0,1]).
}
$$
Our main results concerning $\varphi$ are as follows:
\begin{enumerate}
\item We provide several formulas for the Taylor coefficients of $\varphi$.
\item We show that $\varphi$ is a $C_\infty$ morphism of $C_\infty$ algebras.
\item We prove that $\varphi$ is unique within a certain class of $A_\infty$ morphisms, and, as a consequence, that it coincides with the morphism obtained via homotopy transfer along Dupont's contraction, cf. Subsection \ref{subsection: Whitney forms}.
\end{enumerate}

\subsection{An explicit formula}\label{subsection: varphi}

The aim of this subsection is to make the $A_\infty$ morphism
$$ \varphi: \Omega^*([0,1]) \to C_\infty^*([0,1]),$$
defined as the composition of
$\lambda: \Omega^*([0,1]) \to C_\cup^*([0,1])$ from Subsection \ref{subsection: lambda} and $\log: C_\cup^*([0,1]) \to C_\infty^*([0,1])$ from Subsection \ref{subsection: comparison},
explicit.

\begin{definition}
The {\em descent number} $d_\sigma$ of a permutation $\sigma \in S_n$
is the non-negative integer
$$ d_\sigma := \vert \{ i\in \{1,\dots,n-1\} \textrm{ such that } \sigma(i)>\sigma(i+1)\}\vert.$$

\end{definition}

\begin{theorem}\label{theorem: Eulerian} The higher Taylor coefficients $\varphi_n$, $n\geq2$, of $\varphi$ vanish unless all of the inputs are one-forms, in which case one has
\begin{eqnarray*}
\varphi_n(\bs a_1(t)dt \otimes \cdots \otimes \bs a_n(t)dt) = 
\int\limits_{0\le t_1 \le \cdots \le t_n \le 1} 
\frac{1}{n} \sum_{\sigma \in S_n}\left(\frac{(-1)^{d_\sigma}}{{n-1 \choose d_\sigma}}
a_1(t_{\sigma(1)})\cdots a_n(t_{\sigma(n)})\right)
dt_1 \cdots dt_n \bs dt.
\end{eqnarray*}
\end{theorem}

\begin{proof} Since both the higher Taylor coefficients of $\lambda$ and $\log$ vanish unless all of the inputs are one-forms, the first assertion is clear. When all the inputs are one-forms, by definition of $\log$ and $\lambda$ we have 
\begin{multline*}
\varphi_n(\bs a_1(t)dt \otimes \cdots \otimes \bs a_n(t)dt) =\\= \sum_{m=1}^{n}\frac{(-1)^{m+1}}{m}\sum_{i_1+\cdots + i_m=n}\left(\int_{\Delta_{i_1}}a_1(t_1)\cdots a_{i_1}(t_{i_1})dt_1\cdots dt_{i_1}\cdots\int_{\Delta_{i_k}}a_{n-i_{m}+1}(t_1)\cdots a_{n}(t_{i_m})dt_1\cdots dt_{i_m}\right) \bs dt.
\end{multline*}
According to (the proof of) Proposition \ref{proposition: guggheneim shuffle}, the right hand side of the previous equation equals $\bs dt$ multiplied by the scalar
\begin{eqnarray*}
 \sum_{\sigma \in S_n}\left( \sum_{\substack{i_1+\cdots + i_m=n\\ \sigma \in S(i_1,\dots,i_m)}}\frac{(-1)^{m+1}}{m} \right) \int_{\Delta_n}a_1(t_{\sigma(1)})\cdots a_n(t_{\sigma(n)}) dt_1 \cdots dt_n.
\end{eqnarray*}
Hence the proof is completed by the following lemma. 
\end{proof}

\begin{lemma}\label{lemma: combinatorial}
Given a positive integer $n$ and a permutation $\sigma\in S_n$, we have
$$ 
\sum_{\substack{i_1+\cdots + i_m=n\\ \sigma \in S(i_1,\dots,i_m)}}\frac{(-1)^{m+1}}{m} = \frac{(-1)^{d_\sigma}}{n {n-1 \choose d_\sigma}},$$
where the sum runs over all ordered partitions $i_1+\cdots+ i_m=n$ of $n$ such that $\sigma$ is an $(i_1,\ldots,i_m)$-unshuffle.
\end{lemma}

\begin{proof}
Let us consider partitions $i_1+\cdots +i_m=n$ with a fixed $m$
such that $\sigma \in S(i_1,\dots,i_m)$.
One sees that there are
$$ {n-d_\sigma -1 \choose m-d_\sigma-1}$$
of those. 
Hence we find
$$
\sum_{\substack{i_1+\cdots + i_m=n\\ \sigma \in S(i_1,\dots,i_m)}}\frac{(-1)^{m+1}}{m} = \sum_{m=d_\sigma+1}^n \frac{(-1)^{m+1}}{m}{n-d_\sigma -1 \choose m-d_\sigma-1} =\sum_{j=0}^{n-d_\sigma-1}\frac{(-1)^{d_\sigma+j}}{d_\sigma +j +1}{n-d_\sigma-1 \choose j}. 
$$
The latter sum can be identified with the $n$'th Taylor coefficient
of $\log{(1+z)}(1+z)^{n-d_\sigma-1}$ at $z=0$
and is given by
$$ (-1)^{d_\sigma}\frac{(n-d_\sigma-1)! d_\sigma!}{n!} = \frac{(-1)^{d_\sigma}}{n {n-1\choose d_\sigma}}.$$
\end{proof}

	\begin{definition}\label{remark: Eulerian idempotent} We denote by $C_{n,d}$ the numbers $C_{n,d}:=\frac{(-1)^{d}}{n\binom{n-1}{d}}$, $n\geq1$, $0\leq d<n$: these satisfy the identities \[C_{n,d}=C_{n-1,d}+C_{n,d+1},\qquad C_{n,d}=(-1)^{n+1}C_{n,n-d-1},\] 
as it follows by straightforward computations. For $n\leq 6$, they are given by
\begin{equation*} \begin{tabular}{cccccc} $1$ &&&&& \\ [1ex] $\frac{1}{2}$&$-\frac{1}{2}$&&&& \\ [1ex] $\frac{1}{3}$&$-\frac{1}{6}$&$\frac{1}{3}$&&& \\ [1ex] $\frac{1}{4}$&$-\frac{1}{12}$&$\frac{1}{12}$&$-\frac{1}{4}$&& \\ [1ex] $\frac{1}{5}$&$-\frac{1}{20}$&$\frac{1}{30}$&$-\frac{1}{20}$&$\frac{1}{5}$& \\  [1ex] $\frac{1}{6}$&$-\frac{1}{30}$&$\frac{1}{60}$&$-\frac{1}{60}$&$\frac{1}{30}$&$-\frac{1}{6}$ \end{tabular}
	\end{equation*}
The element $e_n^{[1]}:=\sum_{\sigma\in S_n}C_{n,d_\sigma}\sigma\in\field[S_n]$ of the group algebra of the symmetric group is called the (first) \emph{Eulerian idempotent}, see \cite{Loday,Reutenauer}.\end{definition}
\begin{remark} There is a natural action of $S_n$ on the functions on the $n$-cube, by permuting the variables, and a projector corresponding to $e_n^{[1]}$: then the integrand of Theorem \ref{theorem: Eulerian} is precisely the image of $a_1(t_1)\cdots a_n(t_n)$ under this projector.
\end{remark}
It is well known that $e_n^{[1]}$ is an idempotent of the group algebra, and in fact a \emph{Lie idempotent}. The latter means the following: let $V$ be a vector space and $\overline{T}(V)$ the reduced tensor algebra on $V$, then the mapping 
	\begin{equation}\label{equation: eulerian}
	 E: \overline{T}(V) \to \overline{T}(V), \quad v_1\otimes \cdots \otimes v_n \mapsto \sum_{\sigma \in S_n}C_{n,d_\sigma} v_{\sigma(1)}\otimes \cdots \otimes v_{\sigma(n)}\end{equation}
	is a projector from $\overline{T}(V)$ onto the subspace $\mathrm{Lie}(V)$ spanned by Lie words, i.e. onto the free Lie algebra on $V$, see \cite{Loday,Reutenauer}. Notice that the restriction of $E$ to $n$'fold tensor products is precisely the projector corresponding to $e^{[1]}_n$ under the natural action of $S_n$ on $T^n(V)$. It is not immediately clear how to express $E(v_1\otimes\cdots\otimes v_n)$ as a linear combination of Lie words: one way to do it is to compose $E$ with a second Lie idempotent, for instance the Dynkin idempotent
$$ \gamma: \overline{T}(V) \to \overline{T}(V), \quad v_1\otimes \cdots \otimes v_n \mapsto \frac{1}{n} [v_1,\cdots, [v_{n-1},v_n]\cdots ].$$
Since both $E$ and $\gamma$ are projectors with image $\operatorname{Lie}(V)$, we see that $E=\gamma \circ E$, that is
\begin{equation}\label{equation:E lie} E( v_1\otimes \cdots \otimes v_n ) = \frac{1}{n^2}\sum_{\sigma \in S_n}\frac{(-1)^{d_\sigma}}{{n-1\choose d_\sigma}} [v_{\sigma(1)}, \cdots [v_{\sigma(n-1)}, v_{\sigma(n)}]\cdots].\end{equation} 
We shall make use of the above identity in Section \ref{section: pushforward} below. 	
\\

A choice of basis of $V$ induces a scalar product $(-,-)$ on $\overline{T}(V)$, by imposing that the induced basis of $\overline{T}(V)$ is orthonormal. We consider the adjoint $E^*$ of \eqref{equation: eulerian} with respect to this scalar product: this is independent on the choice of basis, and may be computed explicitly as in the proof of \cite[Theorem 6.3]{Reutenauer} (where $\circledast$ is the shuffle product, cf. the Appendix)
\begin{multline*}\label{equation: E^*} E^*:\overline{T}(V)\to\overline{T}(V), \quad E^*(v_1\otimes\cdots\otimes v_n)= \\= v_1\otimes\cdots\otimes v_n + \sum_{k=2}^{n}\frac{(-1)^{k+1}}{k}\sum_{i_1+\cdots+i_k=n}(v_1\otimes\cdots\otimes v_{i_1})\circledast\cdots\circledast(v_{i_1+\cdots+i_{k-1}+1}\otimes\cdots\otimes v_n).\end{multline*}
By definition of the shuffle product, a straightforward application of Lemma \ref{lemma: combinatorial} yields the more explicit formula
 \[ E^*(v_1\otimes\cdots\otimes v_n) = \sum_{\sigma\in S_n}C_{n,d_\sigma}v_{\sigma^{-1}(1)}\otimes\cdots\otimes v_{\sigma^{-1}(n)}. \]
\begin{proposition}\label{theorem: C_oo} The corestriction $p\varphi:\overline{T}(\bs \Omega^*([0,1]))\to \bs C^*([0,1])$ of the $A_\infty$ morphism $\varphi:\Omega^*([0,1])\to C^*_\infty([0,1])$ equals 
	\[ p\varphi = (p\lambda) \circ E^*, \]
where $p\lambda:\overline{T}(\bs \Omega^*([0,1]))\to\bs C^*([0,1])$ is the corestriction of Gugenheim's $A_\infty$ morphism $\lambda$ and $E^*:\overline{T}(\bs \Omega^*([0,1]))\to\overline{T}(\bs \Omega^*([0,1]))$ is defined as above.
\end{proposition}
\begin{proof} This is a straightforward consequence of the previous formula for $E^*$: if one of the arguments is a zero-form, both the left- and the right-hand side of the claimed identity vanish, otherwise, we see that
	\[ (\lambda_n\circ E^*)(\bs a_1(t)dt\otimes\cdots\otimes \bs a_n(t)dt)=\sum_{\sigma\in S_n}C_{n,d_{\sigma}}\int_{\Delta_n}a_{\sigma^{-1}(1)}(t_1)\cdots a_{\sigma^{-1}(n)}(t_n)dt_1\cdots dt_n =  \]
 	\[ =\sum_{\sigma\in S_n}C_{n,d_{\sigma}}\int_{\Delta_n}a_{1}(t_{\sigma(1)})\cdots a_{n}(t_{\sigma(n)})dt_1\cdots dt_n = \varphi_n(\bs a_1(t)dt\otimes\cdots\otimes \bs a_n(t)dt). \]
\end{proof}

\begin{corollary}\label{corollary: C-infty} The map $\varphi:\Omega^*([0,1])\to C^*_\infty([0,1])$ is a $C_\infty$ morphism.
\end{corollary}
\begin{proof} Recall, cf. Appendix \ref{appendix: C-infty}, that since $\varphi$ is a morphism of $A_\infty$ algebras, we only have to check that the Taylor coefficients $\varphi_n$, $n\geq2$, vanish on the image of the shuffle product
$$ \circledast: \overline{T}(\bs \Omega^*([0,1])\otimes \overline{T}(\bs \Omega^*([0,1]))\to \overline{T}(\bs \Omega^*([0,1]).$$
This follows from the previous proposition, since $E^*$ vanishes on the image of the shuffle product, compare with the proof of \cite[Theorem 6.3]{Reutenauer}.
\end{proof}
We shall give another proof of the previous corollary below, in Subsection \ref{subsection: uniqueness}.

\subsection{A recursive formula}\label{subsection: recursive}

In this subsection we derive an alternative presentation of the $C_\infty$ morphism
$$\varphi: \Omega^*([0,1]) \to C_\infty^*([0,1]),$$
closely related to Magnus expansion (see \cite{Magnus, Iserles-Norsett}), cf. Section \ref{section: pushforward} below.

\begin{definition}\label{definition: IN formula}
For all $n\geq1$ we define maps 
$$\sM_{n}: (\Omega^0([0,1]))^{\otimes n}\to \Omega^0([0,1])$$
as follows:
\begin{itemize}
\item   for $n=1$, we set $\sM_{1}(a_1(t))(s)=\int_0^s a_1(t_1)dt_1$,
\item for $n\geq2$, we apply the recursive formula (where the suspension points inside parentheses are to be filled by the arguments in the order $a_1,\ldots,a_n$, and we denote by $p_j$ the partial sum $p_{j}=\sum_{h\leq j}i_h$)
	\begin{multline*} \sM_{n}(a_1(t)\otimes\cdots\otimes a_n(t))(s)=\sum_{k=1}^{n-1}\frac{B_k}{k!}\sum_{j=0}^k(-1)^{j}{k \choose j} \\ \hspace{0.7cm} \sum_{i_1+\cdots+i_k=n-1}\int_0^s\sM_{i_1}(\cdots)(t_n)\cdots\sM_{i_j}(\cdots)(t_n)a_{p_j+1}(t_n)\sM_{i_{j+1}}(\cdots)(t_n)\cdots\sM_{i_k}(\cdots)(t_n)dt_n.
	\end{multline*}
\end{itemize}
\end{definition}


\begin{definition}\label{definition: sigma_s}
We denote by $(\beta_s)_{s\in\field}$ ($s\in[0,1]$ in the smooth case) the one-parameter family
of maps given by $\beta_s(t)= s\cdot t$.
We refer to the corresponding endomorphisms $\beta_s^*$ of the dg algebra
$\Omega^*([0,1])$ as the {\em scaling morphisms}
and define a one-parameter family of $C_\infty$-morphisms from $\Omega^*([0,1])$ to $C^*_\infty([0,1])$ by setting
$\varphi_s := \varphi \circ \beta_s^*$.
\end{definition}

\begin{theorem}\label{theorem: recursion} The $n$'th Taylor coefficient $\varphi_{s,n}$ of the $C_\infty$ morphism $\varphi_s$ is given by
$$ \varphi_{s,n}(\bs a_1(t)dt \otimes \cdots \otimes \bs a_n(t)dt) =  \sM_{n}(a_1(t)\otimes \cdots \otimes a_n(t))(s) \bs dt.$$
	\end{theorem}
	
	\begin{proof}
We proceed by showing that the family of maps
$$ \nu_{n}: (\Omega^0([0,1])^{\otimes n} \to \Omega^0([0,1])$$
defined by
$\nu_{n}(a_1(t)\otimes \cdots \otimes a_n(t))(s)\bs dt = \varphi_{s,n}(\bs a_1(t)dt\otimes \cdots \otimes \bs a_n(t)dt)$
obeys the same recursion as the family of maps $(\sM_{n})_{n\ge 1}$ from above.	

	Recall that by definition we have $\varphi_{s}=\varphi \circ \beta_s^*$.
	Let $X$ be an arbitrary element of $\overline{T}(\bs \Omega^1([0,1]))$
	and consider the curve
	$$ s \mapsto  \varphi_s(X) \in \overline{T}(\bs C^1([0,1])) \cong \overline{T}(\field).$$
	If we differentiate it with respect to $s$, we find
	$$ \frac{d}{ds} \varphi_s(X)=\varphi\left( \frac{d}{ds}\beta_s^*(X)	\right).$$
	Now suppose we find a one-parameter family of elements $Y_s \in \overline{T}(\bs \Omega^*([0,1]))$ such that $\frac{d}{ds}\beta_s^*X = Q(Y_s)$,
	where $Q$ denotes the codifferential which encodes the dg algebra structure on $\Omega^*([0,1])$.
	We would then conclude that
	$$\frac{d}{ds}(\varphi_s(X))=\varphi(\frac{d}{ds}\beta_s^*(X))=\varphi(Q(Y_s))=M(\varphi(Y_s))$$
	holds, where $M$ is the codifferential on $\overline{T}(\bs C^*([0,1]))$ which encodes the $C_\infty$ algebra structure on $C^*([0,1])$
	from Theorem \ref{theorem: induced structure}.
	
	We now consider $X= \bs a_1(t)dt \otimes \cdots \otimes \bs a_n(t)dt$
	and claim that an appropriate $Y_s$ is given by
	$$ Y_s = -\sum_{j=0}^{n-1} \bs \beta_s^*(a_1(t)dt) \otimes \cdots \otimes \bs \beta_s^*(a_{j}(t)dt) \otimes \bs (ta_{j+1}(st))
	\otimes \bs \beta_s^*(a_{j+2}(t)dt) \otimes \cdots \otimes \bs \beta_s^*(a_n(t)dt).$$
	It is straightforward to show that applying the linear part $Q_1$ of the codifferential $Q$, which encodes the de Rham differential,
	yields $Q_1(Y_s)= X$.
	One checks that the contribution from the quadratic part $Q_{2}$ of the codifferential, which encodes the wedge product, vanishes. This is due to the equality $$\beta_s^*(a_j(t)dt) (ta_{j+1}(st)) = st a_j(st)a_{j+1}(st) dt = (ta_j(st)) \beta_s^*(a_{j+1}(t)dt).$$

From this we infere that
\begin{eqnarray*}
&& \frac{d}{ds}\left(\varphi_s(\bs a_1(t)dt \otimes \cdots \otimes \bs a_n(t)dt \right) =\\
&& \hspace{-0.9cm} -M\varphi \left( \sum_{j=0}^{n-1}
\bs \beta_s^*(a_1(t)dt) \otimes \cdots \otimes \bs \beta_s^*(a_{j}(t)dt) \otimes  \bs (ta_{j+1}(st))
	\otimes \bs \beta_s^*(a_{j+2}(t)dt) \otimes \cdots \otimes \bs \beta_s^*(a_n(t)dt)\right).
\end{eqnarray*}	
Recall from Subsection \ref{subsection: Whitney forms} that the Taylor coefficients of $M$ are only non-zero on tensor products
which contain exactly one factor in $\bs C^0([0,1])$,
while the Taylor coefficients $(\varphi_{n})$
all vanish for $n\ge 2$ whenever one of the factors is a function. Furthermore, we notice that $\varphi_{1}=\pi$ evaluates
on $ta_{j+1}(st)$ to $a_{j+1}(s) t$.
We thus obtain that the projection of
\begin{eqnarray*}
-M\varphi\left( \sum_{j=0}^{n-1}
\bs \beta_s^*(a_1(t)dt) \otimes \cdots \otimes \bs \beta_s^*(a_{j}(t)dt) \otimes \bs (ta_{j+1}(st))
	\otimes \bs \beta_s^*(a_{j+2}(t)dt) \otimes \cdots \otimes \bs \beta_s^*(a_n(t)dt)\right)
	\end{eqnarray*}
to $\bs C^*([0,1])$ equals (by definition of the functions $(\nu_n)_{n\ge1}$)
\begin{eqnarray*}
&&\hspace{-0.75cm}-\sum_{\ell\ge 1}\sum_{p=0}^\ell (-1)^{\ell+1}{\ell \choose p}\frac{B_\ell}{\ell!}
\sum_{j=0}^{n-1}\sum_{n_1+\cdots + n_\ell=n-1}\nu_{n_1}(\cdots)(s)\cdots\nu_{n_p}(\cdots)(s)a_{j+1}(s)
\nu_{n_{p+1}}(\cdots)(s)\cdots \nu_{n_\ell}(\cdots)(s)
\end{eqnarray*}	
times $\bs dt$.

On the other hand,
the projection of
\begin{eqnarray*}
&&\frac{d}{ds}\left(\varphi_{s}(\bs a_1(t)dt \otimes \cdots \otimes \bs a_n(t)dt \right)
\end{eqnarray*}
to $\bs C^*([0,1])$
equals
\begin{eqnarray*}
&&\frac{d}{ds}\left(\varphi_{s,n}(\bs a_1(t)dt \otimes \cdots \otimes \bs a_n(t)dt)\right) = \frac{d}{ds}(\nu_{n}(a_1(t) \otimes \cdots \otimes  a_n(t)) (s)\bs dt)
\end{eqnarray*}
and hence we finally arrive at the recursion
\begin{eqnarray*}
&&\hspace{-0.8cm}\frac{d}{ds}\left(\nu_{n}(a_1(t)\otimes \cdots \otimes a_n(t))(s)\right) =\\
&&\hspace{-0.8cm}= \sum_{\ell\ge 1}\sum_{p=0}^\ell (-1)^{\ell}{\ell \choose p}\frac{B_\ell}{\ell!}
\sum_{j=0}^{n-1}\sum_{n_1+\cdots + n_\ell=n-1}\nu_{n_1}(\cdots)(s)\cdots \nu_{n_p}(\cdots)(s)a_{j+1}(s)
\nu_{n_{p+1}}(\cdots)(s)\cdots \nu_{n_\ell}(\cdots)(s).
\end{eqnarray*}
This is precisely the recursion which the family of maps
$(\sM_{n})_{n\ge 1}$ obeys.

	\end{proof}

\subsection{Some uniqueness results}\label{subsection: uniqueness}
The aim of this section is to show that the $C_\infty$ morphism $\varphi:\Omega^*([0,1])\to C^*([0,1])$ coincides with the $A_\infty$ morphism induced via homotopy transfer along Dupont's contracion, cf. Section \ref{subsection: Whitney forms}. We do so by showing a uniqueness result for $A_\infty$ morphisms satisfying some properties in the following lemma. 
We shall rely heavily on the notations and results from the Appendix.

\begin{lemma}\label{lemma: uniqueness 1} Let $(V,Q_1,\ldots,Q_n,\ldots)$ be an $A_\infty$ algebra, together with a decomposition of $V$ in the direct sum of graded subspaces $V=X\oplus Y$ such that $\bs Y\subset Q_1(\bs X)$.
Let $(W,R_1,\ldots,R_n,\ldots)$ be a second $A_\infty$ algebra and $G, G': V\to W$ two $A_\infty$ morphisms such that
\begin{enumerate}
\item the linear parts of $G$ and $G'$ are equal and
\item whenever there exists $1\leq i \leq n$ such that $v_i\in X$, then $$G_n(\bs v_1\otimes\cdots\otimes \bs v_n)=G'_n(\bs v_1\otimes\cdots\otimes \bs v_n).$$
\end{enumerate}
Under these conditions, the morphisms $G$ and $G'$ coincide.

In particular, if there exists an $A_\infty$ morphism $F:V\to W$ with a given linear part and the property that its higher Taylor coefficients $F_n$, $n\geq2$, vanish whenever at least one of their arguments is in $X$, it is unique.
\end{lemma}

\begin{proof} We have to prove that in the given hypotheses for all $n\geq1$ and $y_1,\ldots,y_n\in Y$ we have $G_{n}(\bs y_1\otimes\cdots\otimes \bs y_n)=G'_n(\bs y_1\otimes\cdots\otimes \bs y_n)$. We use induction, knowing by hypothesis that $G_1=G'_1$. We denote by $Q_n^i$ the composition $\bs V^{\otimes n}\hookrightarrow\overline{T}(\bs V)\xrightarrow{Q}\overline{T}(\bs V)\to \bs V^{\otimes i}$ and similarly for $G^i_n,G'^i_n:\bs V^{\otimes n}\to \bs W^{\otimes i}$: notice that for $i\geq2$ $G^i_n=G'^i_n$ by the inductive hypothesis, since they only depend on $G_1=G'_1,\ldots, G_{n-1}=G'_{n-1}$. Since $G$ is an $A_\infty$ morphism we have the identity
	\[ \sum_{i=1}^n R_i G^i_n =\sum_{j=1}^n G_jQ^j_n, \] 
and similarly for $G'$. Finally, we choose $x_1\in X$ such that $\bs y_1=Q_1(\bs x_1)$, then by the hypotheses of the lemma $G_n(\bs x_1\otimes \bs y_2\otimes\cdots\otimes \bs y_n)=G'_n(\bs x_1\otimes \bs y_2\otimes\cdots\otimes \bs y_n)$, and together with the inductive hypothesis this shows that (notice that $Q_1(\bs Y)=0$) 

\begin{eqnarray*} G_n(\bs y_1\otimes\cdots\otimes \bs y_n)&=&G_nQ_1(\bs x_1\otimes \bs y_2\otimes\cdots\otimes \bs y_n) \\&=& \left( \sum_{i=1}^n R_i G^i_n -\sum_{j=1}^{n-1} G_jQ^j_n\right)(\bs x_1\otimes \bs y_2\otimes\cdots\otimes \bs y_n)  \\&=& \left( \sum_{i=1}^n R_i G'^i_n -\sum_{j=1}^{n-1} G'_jQ^j_n \right)(\bs x_1\otimes \bs y_2\otimes\cdots\otimes \bs y_n)
\\&=&G'_nQ_1(\bs x_1\otimes \bs y_2\otimes\cdots\otimes \bs y_n)=G'_n(\bs y_1\otimes\cdots\otimes \bs y_n).
\end{eqnarray*}
\end{proof}

\begin{corollary}\label{corollary: essential uniqueness}
The $A_\infty$ morphisms
$$ \lambda: \Omega^*([0,1]) \to C^*_{\cup}([0,1]), \quad \exp: C_\infty^*([0,1]) \to C_ \cup^*([0,1]) \quad \textrm{and} \quad 
\varphi: \Omega^*([0,1]) \to C_\infty^*([0,1])$$
are the only $A_\infty$ morphisms with linear parts equal to $\pi$, $\mathrm{id}$ and $\pi$ respectively, and the property that their higher order Taylor coefficients vanish whenever
one of the arguments is a zero-form.
\end{corollary}

\begin{proof}
Apply the last assertion of the preceding lemma to $V=\Omega^*([0,1])$, $X=\Omega^0([0,1])$ and $Y=\Omega^1([0,1])$ (respectively, $V=C^*([0,1])$, $X=C^0([0,1])$ and $Y=C^1([0,1])$).
\end{proof}

\begin{lemma}\label{lemma: uniqueness 2} Under the same hypotheses as in the previous lemma, suppose moreover that $V$ and $W$ are $C_\infty$ algebras. If there exists $F:V\to W$ as in the final claim, then $F$ is a $C_\infty$ morphism.
\end{lemma}
\begin{proof} 
We denote the Taylor coefficients of the $C_\infty$ algebra structure on $V$ and $W$ by $(Q_i)_{i\ge 1}$ and $(R_i)_{i\ge 1}$, respectively.
We have to show $F_n((\bs y_1\otimes\cdots\otimes\bs y_i)\circledast(\bs y_{i+1}\otimes\cdots\otimes\bs y_n))=0$ for all $n\geq1$, $1\leq i<n$ and $y_1,\ldots,y_n\in Y$. The case $n=1$ being empty, we use induction: in particular, we can consider
the morphism of graded coalgebras $F_{<n}:\overline{T}(\bs V)\to\overline{T}(\bs W)$, whose Taylor coefficients
are $(F_{<n})_i = F_i$ for $i<n$ and $(F_{<n})_i=0$ for $i\ge n$, and according to the inductive hypothesis and Lemma \ref{lemma: B4}, this is a morphism of graded bialgebras. We define coderivations $Q_{\geq2}$ on $\overline{T}(\bs V)$ and $R_{\geq2}$ on $\overline{T}(\bs W)$ by declaring their Taylor coefficients to be $(Q_{\geq2})_i=Q_i$, $(R_{\geq2})_i = R_i$ if $i\ge 2$
and $(Q_{\geq2})_1=(R_{\geq2})_1=0$: according to Lemma \ref{lemma: B3} these are biderivations. 
Thus $H:=R_{\geq2}F_{<n}-F_{<n}Q_{\geq2}:\overline{T}(\bs V)\to\overline{T}(\bs W)$ is an $F_{<n}$-biderivation by Lemma \ref{lem:F-derivation}, and in particular it sends the image of the shuffle product in $\overline{T}(\bs V)$ into the image of the shuffle product in $\overline{T}(\bs W)$. As in the proof of the previous lemma we choose $x_1$ such that $Q_1(\bs x_1)=\bs y_1$: we finally compute, denoting by $p:\overline{T}(\bs W)\to \bs W$ the natural projection, that
\begin{eqnarray*} && F_n((\bs y_1\otimes\cdots\otimes\bs y_i)\circledast(\bs y_{i+1}\otimes\cdots\otimes \bs y_n))=\\
&& \hspace{1cm} =(F_nQ_1-R_1F_n)((\bs x_1\otimes \bs y_2\otimes\cdots\otimes\bs y_i)\circledast(\bs y_{i+1}\otimes\cdots\otimes \bs y_n) )\\
&& \hspace{1cm}= \left( \sum_{i=2}^n R_i F^i_n -\sum_{j=1}^{n-1} F_jQ^j_n\right)((\bs x_1\otimes \bs y_2\otimes\cdots\otimes\bs y_i)\circledast(\bs y_{i+1}\otimes\cdots\otimes \bs y_n) )\\ && \hspace{1cm} =pH((\bs x_1\otimes \bs y_2\otimes\cdots\otimes\bs y_i)\circledast(\bs y_{i+1}\otimes\cdots\otimes \bs y_n))\\
&& \hspace{1cm} =0, 
\end{eqnarray*} since $p$ vanishes on the image of the shuffle product. 
\end{proof}

\begin{corollary}\label{corollary: uniqueness of gamma}
$\varphi: \Omega^*([0,1]) \to C^*_\infty([0,1])$
is a morphism of $C_\infty$ algebras.
\end{corollary}

\begin{remark} 
Since $\lambda$, $\exp$, $\log$, $\varphi$ are all compatible with the simplicial structure on $[0,1]$,
we can extend them to morphisms over any $1$-dimensional simplicial complex $T$, and Corollary \ref{corollary: uniqueness of gamma} still holds.
If, moreover, $H_1(T)=0$, we can apply the previous lemmas to  $\Omega^*(T)=\Omega^0(T)\oplus\Omega^1(T)$ and $C^*(T)=C^0(T)\oplus C^1(T)$, respectively, to obtain uniqueness results parallel to Corollary \ref{corollary: essential uniqueness}.
	\end{remark}


\begin{lemma}
For $n>1$ the Taylor coefficients of the $A_\infty$ morphism $\Omega^*([0,1]) \to C_\infty^*([0,1])$
obtained via homological perturbation theory vanish on $n$-fold tensor products
which contain a factor that is a zero-form.
\end{lemma}

\begin{proof}
Given the contraction data $\iota,\pi,h$ from $\Omega^*([0,1])$ to $C^*([0,1])$ as in Section \ref{subsection: Whitney forms}, we denote by \[ H^n_n=\sum_{i=0}^{n-1}\id^{\otimes i}\otimes (-h)\otimes(\iota\pi)^{\otimes n-i-1}:T^n(\bs\Omega^*([0,1]))\to T^n(\bs\Omega^*([0,1])),\] \[ Q^{n-1}_{n}=\sum_{i=0}^{n-2}\id^{\otimes i}\otimes Q_2\otimes\id^{\otimes n-i-2}:T^n(\bs\Omega^*([0,1]))\to T^{n-1}(\bs\Omega^*([0,1])),\] where $Q_2:T^2(\bs\Omega^*([0,1]))\to\bs\Omega^*([0,1])$ is the quadratic part of the codifferential, encoding the wedge product on $\Omega^*([0,1])$, and finally by $\pi_n:T^n(\bs\Omega^*([0,1]))\to \bs C^*([0,1])$ the Taylor coefficients of the $A_\infty$ morphism $\pi_\infty:\Omega^*([0,1])\to C^*_\infty([0,1])$ induced via homotopy transfer. According to the usual perturbation formulas, cf. \cite{Huebschmann-Kadeishvili,Markl}, the maps $\pi_n$ are determined recursively by $\pi_1=\pi$,
\[ \pi_n = \pi_{n-1}Q^{n-1}_n H^n_n\qquad\mbox{for $n\geq2$},\]
and we want to prove that for $n\geq2$ they vanish on tensor products of total degree less than $0$. For $n=2$ we have $\pi_2=\pi Q_2 H^2_2$: if both arguments are functions, this vanishes by degree reasons, while if exactly one argument is a function, it vanishes since $\pi$ vanishes on functions of the form $f(t)h(a(t)dt)$, as $\pi:\Omega^0([0,1])\to C^0([0,1])$ is strictly multiplicative and the image of $h$ is contained in the kernel of $\pi$. For $n\geq 3$ the thesis follows by a straightforward induction, since $Q^{n-1}_n H^n_n:T^n(\bs\Omega([0,1]))\to T^{n-1}(\bs\Omega([0,1]))$ preserves the total degree.\end{proof}

We obtain the following result as an immediate consequence of Lemma \ref{lemma: uniqueness 1}:

\begin{corollary}\label{corollary:homotopy transfer}
The $A_\infty$ morphism $\Omega^*([0,1]) \to C^*_\infty([0,1])$
obtained from homological perturbation theory -- see \cite{Huebschmann-Kadeishvili,Markl} --
 coincides with the $C_\infty$ morphism $\varphi: \Omega^*([0,1]) \to C^*_\infty([0,1])$.
\end{corollary}

We next show that similar uniqueness results hold for the $C_\infty$ algebra structure on $C^*_\infty([0,1])$ and the $C_\infty$ morphism $\mu$.

\begin{proposition} \label{proposition: uniqueness of induced structure}
The unital $C_\infty$ algebra structure on $C^*_\infty([0,1])$, as in Theorem \ref{theorem: induced structure}, is the only one with linear part $m_1(\bs t)=-\bs dt$, quadratic part satisfying $m_2(\bs t\otimes\bs t)=\bs t$ and higher Taylor coefficients vanishing unless precisely one argument is a zero-form. The morphism $\mu$ is the only unital $C_\infty$ morphism from $C^*_\infty([0,1])$ to $\Omega^*([0,1])$ with linear part the inclusion.
\end{proposition}
\begin{proof} We shall denote by $M$ the codifferential on $\overline{T}(\bs C^*([0,1]))$ encoding the $C_\infty$ algebra structure, by $m_{n}$ its Taylor coefficients as in Theorem \ref{theorem: induced structure} and by $M^i_n$ the composition 
	\[ M^i_n:T^n(\bs C^*([0,1]))\hookrightarrow \overline{T}(\bs C^*([0,1]))\xrightarrow{M}\overline{T}(\bs C^*([0,1]))\twoheadrightarrow T^i(\bs C^*([0,1])). \] Notice that the higher coefficients $m_{n+1}$, $n\ge1$, vanish by degree reasons unless precisely one or two of the arguments are zero-forms. To illustrate the result we check the first claim  directly for $m_2$: by the $C_\infty$ property $0=m_2(\bs t\circledast \bs dt)= m_2(\bs t\otimes \bs dt) + m_2(\bs dt\otimes \bs t)$, where $\circledast$ is the shuffle product (cf. Appendix \ref{appendix: C-infty}), hence 
	\[ 2m_2(\bs t\otimes \bs dt)=m_2(\bs t\otimes\bs dt)-m_2(\bs dt\otimes\bs t)=m_2M^2_2(\bs t\otimes\bs t)=-m_1m_2(\bs t\otimes\bs t) =\bs dt,  \]
	from which we get $m_2(\bs t\otimes\bs dt) = \frac{1}{2}\bs dt = - m_2(\bs dt\otimes \bs t)$. Next we assume inductively to have shown the thesis up to a certain $n$. First of all, the $C_\infty$ property implies \[ m_{n+1}((\bs dt)^{\otimes i}\otimes\bs t\otimes (\bs dt)^{\otimes n-i})=(-1)^i\binom{n}{i}m_{n+1}(\bs t\otimes(\bs dt)^{\otimes n}). \] In fact, since this is clear for $i=0$, it follows in general by induction on $i$ and
	\begin{multline*} 0 = m_{n+1}(\bs dt\circledast ((\bs dt)^{\otimes i-1}\otimes\bs t\otimes(\bs dt)^{\otimes n-i})) = \\= im_{n+1}((\bs dt)^{\otimes i}\otimes\bs t\otimes (\bs dt)^{\otimes n-i}) + (n-i+1)m_{n+1}((\bs dt)^{\otimes i-1}\otimes\bs t\otimes (\bs dt)^{\otimes n-i+1}).  \end{multline*}
	Combined with the fact that $M$ is a codifferential, this shows
	\[ (n+1)m_{n+1}(\bs t\otimes(\bs dt)^{\otimes n}) = m_{n+1}M^{n+1}_{n+1}(\bs t^{\otimes 2}\otimes (\bs dt)^{\otimes n-1})=-\sum_{i=1}^n m_iM^i_{n+1}(\bs t^{\otimes 2}\otimes (\bs dt)^{\otimes n-1}). \]
	Since by hypothesis $m_{n+1}(\bs t^{\otimes 2}\otimes (\bs dt)^{\otimes n-1}) = 0$ for $n\ge2$, 
	\[ (n+1)m_{n+1}(\bs t\otimes(\bs dt)^{\otimes n})=-\sum_{i=2}^n m_iM^i_{n+1}(\bs t^{\otimes 2}\otimes (\bs dt)^{\otimes n-1}), \]
	which proves the thesis inductively, as the right hand side only depends on $m_2,\ldots, m_n$.
	
	The claim about $\mu$ is proven similarly. For $n\ge 2$, by degree reasons $\mu_n$ vanishes unless precisely one or none of the arguments are zero-forms. In the latter case, by the $C_\infty$ property $\mu_n((\bs dt)^{\otimes n})=\frac{1}{n!}\mu_n((\bs dt)^{\circledast n})=0$. In the former case, the $C_\infty$ property implies $\mu_{n+1}((\bs dt)^{\otimes i}\otimes\bs t\otimes (\bs dt)^{\otimes n-i})=(-1)^i\binom{n}{i}\mu_{n+1}(\bs t\otimes(\bs dt)^{\otimes n})$ as before. In particular, we see that $(n+1)\mu_{n+1}(\bs t\otimes(\bs dt)^{\otimes n}) = \mu_{n+1}M^{n+1}_{n+1}(\bs t^{\otimes 2}\otimes (\bs dt)^{\otimes n-1})$, and using the facts that $\mu_{n+1}(\bs t^{\otimes 2}\otimes (\bs dt)^{\otimes n-1})=0$ by degree reasons and $\mu$ commutes with the codifferentials, we conclude as before that the right hand side only depends on $\mu_1,\ldots,\mu_n$.
\end{proof}
\begin{remark} In contrast with the final claim of the previous proposition, there can be several {\em $A_\infty$ morphisms} $C^*_\infty([0,1])\to\Omega^*([0,1])$ whose linear part is the inclusion. For instance, a direct verification shows that $F:C^*_\cup([0,1])\to \Omega^*([0,1])$, defined in Taylor coefficients $F_1,\ldots,F_n,\ldots$ by
	\begin{itemize} \item $F$ is unital and $F_1$ is the inclusion;
		\item $F_n((\bs dt)^{\otimes n})= (-1)^{n-1}\bs (t^{n-1}dt)$ for $n\ge 1$;
		\item $F_n((\bs dt)^{\otimes (n-1)}\otimes\bs t)= (-1)^n\bs(t^{n-1}(1-t))$ for $n\ge2$;
		\item For $n\ge2$, $F_n$ vanishes if an argument different from the rightmost one is a zero-form; \end{itemize}
	is a unital $A_\infty$ morphism (which is right inverse to $\lambda$, by a straightforward application of Lemma \ref{lemma:  uniqueness 1}).
Therefore $F\circ\exp:C^*_\infty([0,1])\to\Omega^*([0,1])$ is an $A_\infty$ morphism different from $\mu$, whose linear part is the inclusion. 
\end{remark}

The proof of Proposition \ref{proposition: uniqueness of induced structure}
leads to the following result, which is of independent interest:

\begin{proposition}\label{theorem: automorphisms}

Let $\mathsf{Aut}_\infty(C^*_\infty([0,1]))$ be the group
of unital $C_\infty$ automorphisms of $C^*_\infty([0,1])$,
and $\mathrm{GL}(C^*([0,1])) \cong \field^* \ltimes \field \cong \mathsf{Aff}(\field)$
the group of those automorphisms of the complex $C^*([0,1])$
which map $1$ to itself. The correspondence
\begin{eqnarray*}
r: \mathsf{Aut}_\infty(C^*_\infty([0,1])) &\to & \mathrm{GL}(C^*([0,1])) \cong \mathsf{Aff}(\field) \\
\psi = (\psi_1,\psi_2,\dots) &\mapsto & \psi_1
\end{eqnarray*}
is an isomorphism of groups.
\end{proposition}

\begin{proof}
That the map $r$ is a morphism of groups is clear. The same argument as in the proof of Proposition \ref{proposition: uniqueness of induced structure} shows that a $C_\infty$ morphism with domain $C^*_\infty([0,1])$ is uniquely determined by its linear part, hence $r$ is injective.

To conclude the proof, we have to show that $r$ is surjective.
Let us fix an automorphism $\xi$ of the complex $C^*([0,1])$ mapping $1$ to itself. Evidently, $\xi$ is determined by its value on $t$, given by 
$$ \xi(t) = \alpha \, t + \beta,$$
for $\alpha \in \field^*$ and $\beta \in \field$ two constants.
Our aim is to show that $\xi$ lies in the image of $r$.
We define an automorphism $\rho$ of the unital dg algebra $\Omega^*_\field([0,1])=\field[t] \oplus \field[t]dt$ by declaring its action on the generator $t$ to be
$ \rho(t) = \alpha \, t + \beta$. The composition
$$ \tilde{\xi}: C^*_\infty([0,1]) \stackrel{\mu}{\to} \Omega^*_\field([0,1]) \stackrel{\rho}{\to} \Omega^*_\field([0,1]) \stackrel{\varphi}{\to} C^*_\infty([0,1]),$$
is a unital $C_\infty$ automorphism of $C^*_\infty([0,1])$ 
such that $r(\widetilde{\xi})=\xi$.

\end{proof}

We close this section by sketching a relation with the papers \cite{Adams-Hilton,Majewski}, which was also briefly outlined in the introduction.
\newcommand{\ad}{\operatorname{ad}}
\begin{remark}\label{rem:AW diagonal} We denote by $L([0,1])=\widehat{L}(x,y,a)$ the Lawrence-Sullivan model of the interval: this is the free complete graded Lie algebra on generators $x$ and $y$ in degree one and $a$ in degree zero, and the unique differential such that $x,y$ are Maurer-Cartan elements and $a$ is a gauge equivalence between them, see \cite{Lawrence-Sullivan}, namely, 
\[ d(x)=-\frac{1}{2}[x,x], \quad d(y)=-\frac{1}{2}[y,y],\quad d(a)=\ad_a(y)+\sum_{n\geq0}\frac{B_n}{n!}(\ad_a)^n(y-x), \]
where $\ad_a(-)=[a,-]$ is the adjoint. As observed in the paper \cite{Cheng-Getzler}, this is also the Chevalley-Eilenberg dg Lie algebra associated to the $C_\infty$ algebra $C^*_\infty([0,1])$. We shall denote by $\mathcal{U}(L([0,1]))$ its universal enveloping algebra. Following the notations from the introduction, we shall denote by $\Omega C_*([0,1])$ the cobar construction of the dg coalgebra of normalized chains on $[0,1]$, i.e., the complete tensor algebra $\widehat{T}(x,y,a)$ over generators $x,y,a$ as before, equipped with the differential
	\[d(x)=-x^2,\qquad d(y)=-y^2, \qquad d(a)=(1+a)y-x(1+a).   \] 
Notice that both $\Omega C_*([0,1])$ and $\mathcal{U}(L([0,1]))$ have the same underlying graded algebra $\widehat{T}(x,y,a)$, and only the differentials differ. The $A_\infty$ isomorphism $\exp: C^*_\infty([0,1])\to C^*_\cup([0,1])$ from Subsection \ref{subsection: comparison} yields an isomorphism of dg algebras
\[ \Omega C_*([0,1])\xrightarrow{\cong} \mathcal{U}(L([0,1])),\quad x\mapsto x,\quad y\mapsto y,\quad a\mapsto e^a-1.\]	
In particular, there is an induced cocommutative dg Hopf algebra structure on $\Omega C_*([0,1])$, and it is easy to check that the induced diagonal $\Delta:\Omega C_*([0,1])\to\Omega C_*([0,1])\otimes\Omega C_*([0,1])$ is
\[ \Delta(x)=x\otimes 1+1\otimes x,\quad \Delta(y)=y\otimes 1+1\otimes y,\quad\Delta(a)=a\otimes1+1\otimes a+a\otimes a. \]
Finally, it can be proved, in the spirit of this subsection, that the above $\Delta$ may be characterized as the unique morphism of unital augmented dg algebras satisfying $\Delta(x)=x\otimes 1+1\otimes x, \Delta(y)=y\otimes 1+1\otimes y$: details are left to the interested reader\footnote{We have $\Delta(a)=\sum_{i,j\geq0}r_{i,j}a^i\otimes a^j$ for certain constants $r_{i,j}\in\field$: then $r_{0,0}=0$, since $\Delta$ is a morphism of \emph{augmented} dg algebras,  and one checks, using the fact that $\widehat{T}(x,y,a)$ is a free algebra, that the remaining $r_{i,j}$ are uniquely determined by $\Delta(x)=x\otimes 1+1\otimes x, \Delta(y)=y\otimes 1+1\otimes y$ and the requirement that $\Delta\circ d(a)=\Delta(y-x)+\Delta(a)\Delta(y)-\Delta(x)\Delta(a)=(d\otimes\id+\id\otimes d)\circ\Delta(a)$.}. From this, one can deduce that the diagonal $\Delta$ coincides with the \emph{Alexander-Whitney cobar diagonal} on $\Omega C_*([0,1])$, constructed as in the paper \cite{Adams-Hilton}.
	\end{remark}

\section{Pushforward and the Magnus expansion}\label{section: pushforward}

In this subsection we present implications of our previous results for differential forms on $[0,1]$ with values in a dg algebra $A$ or a dg Lie algebra $\g$, respectively.

\begin{remark}

We will extend the scalars for $\Omega^*([0,1])$ and $C^*([0,1])$ from $\field$ to either a dg algebra $A$ or a dg Lie algebra $\g$. In order for our previous discussion to remain meaningful, we have to guarantee existence and convergence of certain constructions.
Two instances where this works are:
\begin{enumerate}
\item {\em Pro-case}: Assume that $A$ is unital and augmented and that the augmentation ideal $\overline{A}$ is pro-nilpotent. Correspondingly, assume that $\g$ is pro-nilpotent. Then consider polynomial differential forms on $[0,1]$ with values in $A$ or $\g$.
\item {\em Finite-dimensional case}: Assume that $A$ and $\g$ are finite-dimensional
and consider smooth differential forms on $[0,1]$ with values in $A$ or $\g$.
\end{enumerate}
In both cases we obtain dg algebras $\Omega^*([0,1];A)$ and $C^*_\cup([0,1];A)$, an $A_\infty$ algebra $C^*_\infty([0,1];A)$, as well as a dg Lie algebra
$\Omega^*([0,1];\g)$ and an $L_\infty$ algebra $C^*_\infty([0,1];\g)$. The latter two were described in \cite{Mnev} and \cite{Fiorenza-Manetti}.
Observe that, since $C_\cup^*([0,1])$ is not commutative, extension of scalars to $\g$ is not meaningful in this case (within the world of algebras).
\end{remark}

\subsection{Forms with values in a dg algebra}\label{subsection: coefficients dg algebra}

We first consider extension by a unital dg algebra $A$.
The family of $C_\infty$ quasi-isomorphisms $\varphi_s: \Omega^*([0,1]) \to C_\infty^*([0,1])$ extends to a one-parameter family of $A_\infty$ quasi-isomorphism
$$\varphi_s: \Omega^*([0,1];A) \to C_\infty^*([0,1];A),$$
see Definition \ref{definition: sigma_s}. The explicit formulas from Theorem \ref{theorem: Eulerian} and \ref{theorem: recursion} remain valid in this setting, i.e. they are compatible with scalar extension by $A$ (essentially, because they keep the arguments in order). Notice however that Proposition \ref{theorem: C_oo} fails in the non-commutative case.

\begin{definition}
The pushforward along $\varphi_s$ is the mapping
\begin{eqnarray*}
(\varphi_s)_*: \bs \big(\Omega^0([0,1];A^1)\oplus \Omega^1([0,1];A^0)\big) &\to & \bs \big(C_\infty^0([0,1];A^1)\oplus C_\infty^1([0,1];A^0\big)\cong \bs A^1\oplus\bs A^1 \oplus A^0
\\
f(t) + a(t)dt &\mapsto & \sum_{n\ge 1} \varphi_{s,n}(\bs (f(t) + a(t)dt)\otimes \cdots \otimes \bs (f(t)+a(t)dt)).
\end{eqnarray*}
\end{definition}

\begin{remark}
Since $\varphi_{s,n}$ vanishes for $n>1$ whenever one of the inputs
is a zero-form, we find
\begin{eqnarray*}
\sum_{n\ge 1} \varphi_{s,n}(\bs (f(t) + a(t)dt)\otimes \cdots \otimes \bs (f(t)+a(t)dt)) =\\
 \bs(f(s)t + f(0)(1-t)) + \sum_{n\ge 1} \varphi_{s,n}(\bs a(t)dt\otimes \cdots \otimes \bs a(t)dt).
\end{eqnarray*}
Therefore, we see that the essential information is 
the restriction of $(\varphi_s)_*$ to $\bs \Omega^1([0,1];A^0)$.
\end{remark}

We remark that in the pro-case, we have to restrict the domain of definition of $(\varphi_s)_*$ to
$$\bs \big(\Omega^0([0;1];A^1)\oplus \Omega^1([0,1];\overline{A}^0)\big),$$
i.e. we have to require the one-forms to take values in the augmentation ideal $\overline{A}^0$ of $A^0$.
The reason is that this guarantees that the potentially infinite series in the definition of $(\varphi_s)_*$ is well-defined.
Whenever we consider the pro-case, we will from now on apply this restriction.
\\

The following result was
established independently by Burghart-Mn\"ev-Steinebrunner in \cite{BMS}.

\begin{proposition}\label{proposition: ODE1}
For a given $a(t)dt \in \Omega^1([0,1];A^0)$, consider 
the curve
$$ [0,1]\to A^0, \quad s \mapsto A(s)\bs dt:=(\varphi_s)_*(\bs a(t)dt).$$
Its exponential
$$ e^{A(s)}:=1_A + \sum_{k\ge 0}\frac{1}{k!} (A(s))^{k}.$$
satisfies the differential equation
$$\frac{d}{ds} e^{A(s)} = e^{A(s)} \, a(s), \qquad e^{A(0)}= 1_A.$$
\end{proposition}

\begin{proof}
Since pushforward is compatible with composition of morphisms,
we find that
$$ \exp_*\circ (\varphi_s)_* = \exp_* \circ \log_* \circ (\lambda_s)_* = (\lambda_s)_*,$$
where $\lambda_s = \lambda \circ \beta_s^*$,
with $\beta_s^*$ being
the scaling morphism from Definition \ref{definition: sigma_s}.
We therefore have $e^{A(s)} = 1_A + (\lambda_s)_*(\bs a(t)dt)$.

The pushforward along $\lambda_s$ is given by
$$ (\lambda_s)_*(\bs a(t)dt) = \sum_{n\ge 1} \int\limits_{0\le t_1\le \cdots \le t_n\le s} a(t_1)\cdots a(t_n) dt_1\cdots dt_n.$$
Differentiation with respect to $s$ yields
$$ \frac{d}{ds}\left( 1_A + (\lambda_s)_*(a(t)dt)\right) =  \big(1_A + (\lambda_s)_*(a(t)dt)\big) a(s) $$
and for $s=0$, we have $(1_A + (\lambda_0)_*(a(t)dt)) = 1_A$.  
This concludes the proof.
\end{proof}

\begin{remark} \label{remark: pushforward associative}
\hspace{0cm}
\begin{enumerate}
\item By Theorem \ref{theorem: Eulerian}, we can write
the pushforward along $\varphi$ as
\begin{eqnarray*}
(\varphi_s)_*(\bs a(t)dt) = \left( \sum_{n\ge 1} \frac{1}{n}\int\limits_{0\le t_1 \le \cdots \le t_n \le s} 
 \sum_{\sigma \in S_n}\left(\frac{(-1)^{d_\sigma}}{{n-1 \choose d_\sigma}}
a(t_{\sigma(1)})\cdots a(t_{\sigma(n)})\right)
dt_1 \cdots dt_n\right) \bs dt.
\end{eqnarray*}

\item Alternatively, we may use Theorem \ref{theorem: recursion} to describe the pushforward along $\varphi_s$ as follows. We define maps 
$\sM_{n}: \Omega^0([0,1];A^0)^{\otimes n}\to \Omega^0([0,1],A^0)$ recursively as in Definition \ref{definition: IN formula}. Given $a(t)dt \in \Omega^1([0,1];A^0)$, we simplify the notations and put $\sM_k(s):=\sM_k(a(t)^{\otimes k})(s)$, $\sM_\infty(s):=\sum_{k\geq1}\sM_k(s)$. Then, according to Theorem \ref{theorem: recursion},
$$\sM_\infty(s) \bs dt =\sum_{k\ge 1} (\sM_k(s) \bs dt) := \sum_{n\ge 1} \varphi_{s,n}(\bs a(t)dt \otimes \cdots \otimes \bs a(t)dt) =(\varphi_s)_*(\bs a(t)dt).$$ Differentiating the defining recursion for the maps $\sM_k$, we find
		\begin{eqnarray*} \frac{d}{ds}\sM_\infty(s)&=&a(s)+\sum_{k\geq1}\frac{B_k}{k!}\sum_{j=0}^k(-1)^{j}{k \choose j} \sum_{i_1,\ldots,i_k\geq1}\sM_{i_1}(s)\cdots\sM_{i_j}(s)a(s)\sM_{i_{j+1}}(s)\cdots\sM_{i_k}(s)\\
		&=& \sum_{k\geq0}\frac{B_k}{k!}\sum_{j=0}^k(-1)^{j}\binom{k}{j}\sM_\infty(s)^j a(s)\sM_\infty(s)^{k-j}
		\\&=&\sum_{k\geq0}\frac{B_k}{k!}\left[\cdots\left[a(s),\sM_\infty(s)\right]\cdots,\sM_\infty(s)\right],\end{eqnarray*}		
which is equivalent to		\[
	\sum_{k\geq0}\frac{1}{(k+1)!}\left[\cdots\left[\frac{d}{ds}\sM_\infty(s),\sM_\infty(s)\right]\cdots,\sM_\infty(s)\right]=	
		  a(s).\]
		According to a classical result by Hausdorff, compare with \cite[Theorem 2.1]{Iserles-Norsett}, this shows that $e^{\sM_\infty(s)}$ is the solution to the differential equation $\frac{d}{ds} e^{\sM_\infty(s)} = e^{\sM_\infty(s)} a(s)$ with initial condition $e^{\sM_\infty(0)}=1_A$, and provides another proof of Proposition \ref{proposition: ODE1}. 
		\end{enumerate}
	\end{remark}

\subsection{Forms with values in a dg Lie algebra}\label{subsection: coefficients dg Lie algebra}

For $\g$ a dg Lie algebra, we obtain a one-parameter family of $L_\infty$ quasi-ismorphisms
$$\varphi_s: \Omega^*([0,1];\g) \to C^*_\infty([0,1];\g)$$
from $\varphi_s: \Omega^*([0,1])\to C^*_\infty([0,1])$
by extension of scalars (cf. \cite{Miller} for the defintion of scalar extension of a $C_\infty$ algebra by a dg Lie algebra). By compatibility between scalar extension and homotopy transfer, together with Corollary \ref{corollary:homotopy transfer}, this is the same as the composition of the scaling morphism $\beta_s^*$ and the $L_\infty$ morphism induced via homotopy transfer along the obvious extension of Dupont's contraction (cf. \cite{Fiorenza-Manetti}).

We denote the universal enveloping dg algebra of $\g$ by $\mathcal{U}(\g)$.
By compatibility with the symmetrization functor from $A_\infty$ algebra to $L_\infty$ algebras, $\varphi_s$ may also be characterized by the commutative diagram of $L_\infty$ algebras and $L_\infty$ morphisms, 
\[\xymatrix{ \Omega^*([0,1];\g)\ar@{^(->}[r]\ar[d]_-{\varphi_s}   & \Omega^*([0,1];\U(\g))\ar[d]^-{\operatorname{sym}(\varphi_s)} \\ C^*_\infty([0,1];\g)\ar@{^(->}[r] & C^*_\infty([0,1];\U(\g))   }    \]  
where the horizontal arrows are the strict inclusions and the right vertical arrow is the symmetrization of the $A_\infty$ morphism studied in the previous subsection. For convenience, let us define maps
\begin{eqnarray*}
 \sM_{n}: \bigodot^n(\Omega^0([0,1];\g) &\to& \Omega^0([0,1];\g) 
 \end{eqnarray*}
 by setting
$\sM_{n}(l_1(t)\odot \cdots \odot l_n(t))(s) \bs dt := \varphi_{s,n}(\bs l_1(t)dt \odot \cdots \odot \bs l_n(t)dt)$.

\begin{theorem}\label{theorem: L-infty morphism}
\begin{enumerate}
\item
 The maps $(\sM_{n})_{n\ge 1}$ are given by 
\begin{eqnarray*}
&& \hspace{-1cm}\sM_{n}(l_1(t)\odot \cdots \odot  l_n(t))(s) =\\
&& \hspace{-1.5cm}\int\limits_{0\le t_1\le \cdots \le t_n\le s}\left(\frac{1}{n^2}\sum_{\sigma,\tau \in S_n}\varepsilon(\tau)\frac{(-1)^{d_\sigma}}{{n-1 \choose d_\sigma}} [l_{\tau(1)}(t_{\sigma(1)}),\cdots,[l_{\tau(n-1)}(t_{\sigma(n-1)}),l_{\tau(n)}(t_{\sigma(n)})]\cdots ]\right)dt_1\cdots dt_n ,
\end{eqnarray*}
where $\varepsilon(\tau)$ is the Koszul sign associated to $\tau$,
i.e. the sign given by  $ l_1(t)\odot\cdots\odot l_n(t)=\varepsilon(\tau) l_{\tau(1)}(t)\odot\cdots\odot l_{\tau(n)}(t)$. 
\item Equivalently, we may define the maps $(\sM_n)$ recursively by putting $\sM_{1}(l_1(t))(s)=\int_0^s l_1(t_1)dt_1$ for $n=1$,  and for $n>1$
\begin{eqnarray*}
&& \hspace{-0.5cm}\sM_{n}(l_1(t)\odot\cdots\odot l_n(t))(s)= \\
&&  =\sum_{k=1}^{n-1}(-1)^k\frac{B_k}{k!}\sum_{i_1+\cdots+i_k=n-1}\sum_{\sigma\in S_n}\varepsilon(\sigma)\int_0^s\left[\sM_{i_1}(\cdots)(t_n),\cdots\left[\sM_{i_k}(\cdots)(t_n),l_{\sigma(n)}(t_n)\right]\cdots\right]dt_n,
\end{eqnarray*}
where the suspension points inside $\sM_{i_1}(\cdots),\ldots,\sM_{i_k}(\cdots)$ have to be filled by the arguments in the order $l_{\sigma(1)}(t),\ldots,l_{\sigma(n-1)}(t)$.
\end{enumerate}
\end{theorem}

\begin{proof} The first explicit presentation follows by symmetrizing the formulas for the $A_\infty$-morphism $\varphi_s:\Omega^*([0,1];\mathcal{U}(\g))\to C^*_\infty([0,1];\mathcal{U}(\g))$ coming from Theorem \ref{theorem: Eulerian}, where now the arguments $l_i(t)dt$ are elements in $\Omega^1([0,1];\g^0)\subset \Omega^1([0,1];\mathcal{U}^0(\g))$: we see that $\varphi_{s,n}(l_1(t)dt\odot\cdots\odot l_n(t)dt)$ is the integral over the $n$'th simplex of the image of \[\sum_{\tau\in S_n}\varepsilon(\tau)l_{\tau(1)}(t_1)\cdots l_{\tau(n)}(t_n)dt_1\cdots dt_n\in\Omega^n([0,1]^{\times n};\mathcal{U}^0(\g))\] under the Eulerian projector $E:\mathcal{U}^0(\g)\to\g^0$. We recall, compare \cite{Reutenauer}, that the latter may also be understood as the composition $E=p\circ \operatorname{PBW}^{-1}$ of the inverse of the Poincar\'e-Birkhoff-Witt isomorphism $\operatorname{PBW}:S(\g)\to\mathcal{U}(\g)$ and the natural projection $p:S(\g)\to\g$. Finally, we get the desired formula for $\sM_n$ by composing $E$ with the Dynkin idempotent, as we did in formula \eqref{equation:E lie} on page \pageref{equation:E lie}.
		
The claimed recursive presentation for the maps $\sM_n$ is precisely the one we get, after symmetrization, from the corresponding one in the $A_\infty$ case coming from Definition \ref{definition: IN formula}, as it follows by straightforward computations, keeping in mind the formula
	\begin{multline*} \sum_{\sigma\in S_k}\varepsilon(\sigma) [ x_{\sigma(1)},[\cdots[x_{\sigma(k)}, y]\cdots] ]= \\ =\sum_{\sigma\in\ S_k} \varepsilon(\sigma) \sum_{j=0}^{k}(-1)^{k-j+\sum_{h>j}|y||x_h|}\binom{n}{j}x_{\sigma(1)}\cdots x_{\sigma(j)}yx_{\sigma(j+1)}\cdots x_{\sigma(k)}, \end{multline*}
valid in any associative graded algebra. Thus, the second claim follows by Theorem \ref{theorem: recursion}.

\end{proof}

\begin{remark} The first few instances of the previous recursion are
\[\sM_{1}(l_1(t))(s)=\int_0^s l_1(t_1)dt_1,\]
\[\sM_{2}(l_1(t)\odot l_2(t))(s)=\sum_{\sigma\in S_2}\varepsilon(\sigma)\frac{1}{2}\int_0^s\left[\int_0^{t_2}l_{\sigma(1)}(t_1)dt_1,l_{\sigma(2)}(t_2)\right]dt_2,\]
\begin{multline*}	\sM_{3}(l_1(t)\odot l_2(t)\odot l_3(t))(s)=\sum_{\sigma\in S_3}\varepsilon(\sigma) \frac{1}{4}\int_0^s \left[\int_0^{t_3}\left[\int_0^{t_2} l_{\sigma(1)}(t_1)dt_1, l_{\sigma(2)}(t_2)\right]dt_2,l_{\sigma(3)}(t_3)\right]dt_3+\\
+\sum_{\sigma\in S_3}\varepsilon(\sigma)\frac{1}{12}\int_0^s \left[ \int_0^{t_3}l_{\sigma(1)}(t_1)dt_1,\left[\int_0^{t_3}l_{\sigma(2)}(t_2)dt_2,l_{\sigma(3)}(t_3)\right]\right]dt_3.
\end{multline*}

For the pushforward, we find the following general recursion, where for simplicity we put $\sM_n(s)=\sM_n(l(t)^{\odot n})(s)$, $\sM_\infty(s)=\sum_{n\geq1}\frac{1}{n!}\sM_n(s)$,
\begin{eqnarray*}
	&& (\varphi_{s})_*(l(t)dt) = \sM_\infty(s)\bs dt = \sum_{n\ge 1}\frac{1}{n!} \sM_{n}(s)\bs dt = \\
	&& \hspace{0.5cm} = \left(\int_0^s l(t_1)dt_1 + \sum_{n\ge 2}\sum_{k=1}^{n-1}(-1)^k\frac{B_k}{k!}\sum_{i_1+\cdots+i_k=n-1}\int_0^s\left[\sM_{i_1}(t_n),\cdots\left[\sM_{i_k}(t_n),l(t_n)\right]\cdots\right]dt_n\right)\bs dt.
\end{eqnarray*}
Modulo the switch from $[\cdot,\cdot]$ to the opposite bracket $[x,y]^{\mathrm{op}}:=[y,x]$, this is precisely the recursive expansion given by Magnus, see \cite{Magnus,Iserles-Norsett}, for the solution of the differential equation $\frac{d}{ds}e^{\sM_\infty(s)}=e^{\sM_\infty(s)}l(s)$ in the enveloping algebra $\mathcal{U}(\g^0)$, compare with Proposition \ref{proposition: ODE1}.
By the previous theorem, we also find
\begin{equation*}
 \sM_\infty(s)
	=\sum_{n\ge 1} \int\limits_{0\le t_1\le \cdots \le t_n\le s}\left(\frac{1}{n^2}\sum_{\sigma \in S_n}\frac{(-1)^{d_\sigma}}{{n-1 \choose d_\sigma}} [l(t_{\sigma(1)}),\cdots,[l(t_{\sigma(n-1)}),l(t_{\sigma(n)})]\cdots ]\right)dt_1\cdots dt_n.
\end{equation*}
This formula for the Magnus expansion was found by Mielnik and Plaba\'nski \cite{Mielnik-Plebanski}.

\end{remark}

\appendix

\section{Review of $A_\infty$ and $L_\infty$ algebras}\label{appendix: A-infty and L-infty}

We briefly describe our terminology and notations concerning $A_\infty$
and $L_\infty$ algebras. In the next section we shall review in more detail
some results concerning $C_\infty$ algebras.

\begin{itemize}
	\item The suspension endofunctor $\bs$ maps a graded vector space $V$
	to its suspension $\bs V$, whose component $(\bs V)^i$ in degree $i\in \mathbb{Z}$ is $V^{i+1}$.
	\item $\overline{T}(V)=\bigoplus_{n\ge 1}T^n(V)$ denotes the reduced tensor coalgebra on a graded vector space, with the deconcatenation coproduct $\overline{\Delta}:\overline{T}(V)\to \overline{T}(V)\otimes\overline{T}(V)$,
	$$\overline{\Delta}(x_1\otimes \cdots \otimes x_n) = \sum_{i=1}^{n-1}(x_1\otimes \cdots \otimes x_i)\otimes (x_{i+1}\otimes \cdots \otimes x_n).$$
	It is the cofree object over $V$ in the category of coassociative, locally conilpotent (i.e., the union of the kernels of the iterated coproducts is exhaustive) graded coalgebras.
	\item We denote by $S_n$ the $n$'th symmetric group. Given an integer $n\ge1$ and an ordered partition $i_1+\cdots+i_k=n$, we denote by $S(i_1,\ldots,i_k)\subset S_n$ the set of \emph{$(i_1,\ldots,i_k)$-unshuffles}, i.e., permutations $\sigma\in S_n$ such that $\sigma(i)<\sigma(i+1)$ for $i\neq i_1, i_1+i_2,\ldots, i_1+\cdots+ i_{k-1}$.
	\item The symmetric group $S_n$ acts on $T^n(V)$ by $\sigma(x_1\otimes\cdots\otimes x_n)=\varepsilon(\sigma)x_{\sigma(1)}\otimes\cdots\otimes x_{\sigma(n)}$, where $\varepsilon(\sigma)=\varepsilon(\sigma;x_1,\ldots,x_n)$ is the usual \emph{Koszul sign}. 
	We denote the space of coinvariants either by $S^n(V)$ or by $\bigodot^n(V)$, and by $x_1\odot\cdots\odot x_n$ the image of $x_1\otimes\cdots\otimes x_n$ under the natural projection $T^n(V)\to S^n(V)$. The reduced symmetric coalgebra over $V$ is the space $S(V)=\bigoplus_{n\ge 1}S^n(V)$, with the unshuffle coproduct
	$$\overline{\Delta}(x_1\odot \cdots \odot x_n) = \sum_{i=1}^{n-1}\sum_{\sigma\in S(i,n-i)}\varepsilon(\sigma)(x_{\sigma(1)}\odot \cdots \odot x_{\sigma(i)})\otimes (x_{\sigma(i+1)}\odot \cdots \odot x_{\sigma(n)}).$$
	This is the cofree, coassociative, cocommutative and locally conilpotent graded coalgebra over $V$.
	\item 
	Let $(C,\Delta)$ be a graded coalgebra.
	A map $Q: (C,\Delta)\to (C,\Delta)$ of degree $1$ is a codifferential if $Q\circ Q = 0$ and $\Delta \circ Q =(Q\otimes \mathrm{id} + \mathrm{id}\otimes Q)\circ \Delta$ hold true.
	\item An $A_\infty$ algebra structure on a graded vector space $V$
	is a codifferential $Q$ of the graded coalgebra $(\overline{T}(\bs V),\overline{\Delta})$.
	Similarly, an $L_\infty$ algebra structure on $V$ is a codifferential $Q$
	of the graded coalgebra $(\overline{S}(\bs V),\overline{\Delta})$.
	\item A morphism of $A_\infty$ algebras from $A_\infty$ algebra $V$ to $A_\infty$ algebra $W$ is a morphism of the corresponding
	dg coalgebras $F: (\overline{T}(\bs V),\overline{\Delta},Q_V) \to (\overline{T}(\bs W),\overline{\Delta},Q_W)$.
	In the same manner one defines morphisms of $L_\infty$ algebras.
	\item An $A_\infty$ algebra structure $Q$ on $V$ is determined by its Taylor coefficients $(Q_n)_{n\ge 1}$, which are the maps given by
	$$ 
	\xymatrix{
		T^n(\bs V) \ar[r] & \overline{T}(\bs V) \ar[r]^Q & \overline{T}(\bs V) \ar[r]^(0.4){p}& T^1(\bs V) \cong \bs V.
	}$$
	Moreover, a morphism $F$ of $A_\infty$ algebras from $V$ to $W$
	is determined by its Taylor coefficients
	$F_n: T^n(\bs V) \to \bs W$, which are defined in the same manner as
	the Taylor coefficients of an $A_\infty$ algebra structure.
	\item Similarly, an $L_\infty$ algebra structure $Q$ on $V$
	is determined by its Taylor coefficients
	$Q_n: \bigodot^n(\bs V) \to \bs V$, for $n\ge 1$, and a $L_\infty$ algebra morphism $F$ from $V$ to $W$ is determined by its Taylor coefficients
	$F_n: \bigodot^n(\bs V) \to \bs W$.
	\item A morphism of $A_\infty$ algebras, respectively $L_\infty$ algebras,
	is called a quasi-isomorphism if its first Taylor coefficient induces
	an isomorphism on cohomology.
	\item The category of dg algebras embeds into the category of $A_\infty$ algebra via the embedding
	$$ (A,\cdot,d) \mapsto (\overline{T}(\bs A),Q),$$
	where $Q$ is the coderivation whose non-trivial Taylor coefficients are
	$Q_1(\bs a) = - \bs (da)$ and $Q_2(\bs a\otimes \bs b)=(-1)^{|a|}\bs (a \cdot b)$.
	Similar formulas define an embedding of the category of dg Lie algebras
	into the category of $L_\infty$ algebras.
	\item The forgetful functor from dg associative algebras to dg Lie algebras admits the following higher generalization. Given a graded vector space $V$, we denote by $\operatorname{sym}_n$, $n\geq1$,  the maps \[ \operatorname{sym}_n:S^n(\bs V)\to T^n(\bs V),\quad\bs x_1\odot\cdots\odot \bs x_n\mapsto\sum_{\sigma\in S_n}\varepsilon(\sigma)\bs x_{\sigma(1)}\otimes\cdots\otimes \bs x_{\sigma(n)}.\] If $Q_n:T^n(\bs V)\to\bs V$, $n\geq1$, are the Taylor coefficients of an $A_\infty$ algebra structure on $V$, then the $Q_n\circ\operatorname{sym}_n:S^n(\bs V)\to\bs V$ are the Taylor coefficients of an $L_\infty$ algebra structure $\operatorname{sym}(Q)$ on $V$. Similarly, if $F_n:T^n(\bs V)\to\bs W$ are the Taylor coefficients of an $A_\infty$ morphism $F:(V,Q_V)\to (W,Q_W)$, then $F_n\circ\operatorname{sym}_n:S^n(\bs V)\to\bs W$ are the Taylor coefficients of an $L_\infty$ morphism $\operatorname{sym}(F):(V,\operatorname{sym}(Q_V))\to(W,\operatorname{sym}(Q_W))$. This defines the \emph{symmetrization functor} from the category of $A_\infty$ algebras to the one of $L_\infty$ algebras.
\end{itemize}

\section{Review of $C_\infty$ algebras}\label{appendix: C-infty}

$C_\infty$ algebra structures are $A_\infty$ algebra structures which are compatible with the shuffle product on the reduced tensor coalgebra. To be precise,
the reduced tensor coalgebra $(\overline{T}(V),\overline{\Delta})$
can be equipped with the structure of a graded bialgebra by introducing the shuffle product
$$ (v_1\otimes \cdots \otimes v_p)\circledast (v_{p+1}\otimes \cdots \otimes v_n) = \sum_{\sigma \in S(p,q)}\varepsilon(\sigma)v_{\sigma^{-1}(1)}\otimes \cdots \otimes v_{\sigma^{-1}(n)},$$
where $S(p,q)$ is the set of $(p,q)$-unshuffles, i.e. a
permutation $\sigma$ of $\{1,\dots,n\}$ such that
$\sigma(i)<\sigma(i+1)$ for all $i\neq p$.

\begin{definition} A $C_\infty$ algebra structure on a graded space $V$ is a dg bialgebra structure $Q:\overline{T}(\bs V)\to \overline{T}(\bs V)$ on the graded bialgebra $(\overline{T}(\bs V),\overline{\Delta},\circledast)$. A $C_\infty$ morphism $F:V\to W$ between $C_\infty$ algebras $V$ and $W$ is a morphism of dg bialgebras $F:\overline{T}(\bs V)\to\overline{T}(\bs W)$.
\end{definition}

Let $(C,\Delta_C,m_C)$ and $(D,\Delta_D,m_D)$ be graded bialgebras with coproducts $\Delta_C$, $\Delta_D$ and products $m_C$, $m_D$ respectively. Recall that given a morphism $F:(C,\Delta_C)\to(D,\Delta_D)$ of graded coalgebras, a linear map $R:C\to D$ is an $F$-coderivation if it satisfies the identity $\Delta_DR=(R\otimes F+F\otimes R)\Delta_C$. Similarly, given a morphism of graded algebras $F:(C,m_C)\to(D,m_D)$, a linear map $R:C\to D$ is an $F$-derivation if it satisfies the identity $Rm_C=m_D(R\otimes F+F\otimes R)$. Finally, given a morphism of graded bialgebras $F:(C,\Delta_C,m_C)\to(D,\Delta_D,m_D)$, a linear map $R:C\to D$ is an $F$-biderivation if it is both an $F$-coderivation and an $F$-derivation. When $F=\operatorname{id}_C$ we recover the usual definition of a (resp.: co, bi)derivation on $C$. The proof of the following lemma is a straightforward verification.
\begin{lemma}\label{lem:F-derivation} Given a morphism of (resp.: co, bi)algebras $F:C\to D$ and (resp.: co, bi)derivations $Q:C\to C$, $Q':D\to D$, then the maps $FQ,Q'F:C\to D$ are $F$-(resp.: co, bi)derivations.
\end{lemma}

We say that a graded coalgebra $(C,\Delta_C)$ is locally conilpotent if $C=\bigcup_{n\geq1}\operatorname{ker}(\Delta_C^n)$, where $\Delta_C^n:C\to C^{\otimes n+1}$ is the iterated coproduct. Recall that $(\overline{T}(V),\overline{\Delta})$ is the cofree locally conilpotent graded coalgebra over $V$: in particular, if $C$ is locally conilpotent every morphism of graded coalgebras $F:C\to\overline{T}(V)$ (resp.: every $F$-coderivation $R:C\to\overline{T}(V)$) is determined by its corestriction $pF:C\to V$ (resp.: $pR:C\to V$), where we denote by $p:\overline{T}(V)\to V$ the natural projection. This applies to $C=\overline{T}(V)\otimes\overline{T}(V)$, equipped with the induced (locally conilpotent) coalgebra structure: in particular, the shuffle product $\circledast:\overline{T}(V)\otimes\overline{T}(V)\to\overline{T}(V)$ is the only morphism of graded coalgebras with vanishing corestriction $0=p\circledast:\overline{T}(V)\otimes\overline{T}(V)\xrightarrow{\circledast}\overline{T}(V)\xrightarrow{p} V$.

\begin{lemma}\label{lemma: B3} A coderivation $Q:\overline{T}(V)\to\overline{T}(V)$ of a reduced tensor coalgebra is also a derivation with respect to the shuffle product $\circledast$ if and only if its Taylor coefficients $Q_n:V^{\otimes n}\to V$ vanish on the image of $\circledast$.\end{lemma}

\begin{proof} We have to show $Q\circledast=\circledast(Q\otimes\operatorname{id}+\operatorname{id}\otimes Q)$. Since both the left and the right hand side are $\circledast$-coderivations by Lemma \ref{lem:F-derivation}, it suffices to show that they have the same corestriction: as $p\circledast=0$, this happens if and only if the composition $\overline{T}(V)\otimes\overline{T}(V)\xrightarrow{\circledast}\overline{T}(V)\xrightarrow{Q}\overline{T}(V)\xrightarrow{p} V$ also vanishes.\end{proof}
\begin{lemma}\label{lemma: B4} A morphism of graded coalgebras $F:(\overline{T}(V),\overline{\Delta})\to(\overline{T}(W),\overline{\Delta})$ is also a morphism of graded bialgebras if and only if its Taylor coefficients $F_n:V^{\otimes n}\to W$ vanish on the image of the shuffle product.
\end{lemma}\begin{proof} As for the previous lemma, the two morphisms of graded locally conilpotent coalgebras $\circledast(F\otimes F), F\circledast:\overline{T}(V)\otimes\overline{T}(V)\to\overline{T}(W)$ coincide if and only if they have the same corestriction if and only if the composition $\overline{T}(V)\otimes\overline{T}(V)\xrightarrow{\circledast}\overline{T}(V)\xrightarrow{F}\overline{T}(W)\xrightarrow{p} W$ vanishes. 
\end{proof}
Given an $A_\infty$ algebra $V$, whose Taylor coefficients are $(Q_i)_{i\ge 1}$, and contraction data $$(\xymatrix{\bs W\ar@<2pt>[r]^-{F_1}& \bs V\ar@<2pt>[l]^-{G_1}},K).$$
The usual $A_\infty$ homotopy transfer theorem -- see \cite{Huebschmann-Kadeishvili,Markl} -- tells us that the maps $R_n:\bs W^{\otimes n}\to \bs W$ (where $R_1$ is the differential on $\bs W$) and $F_n:\bs W^{\otimes n}\to \bs V$, defined recursively by 
\begin{equation}\label{eq:homotopytransfer} R_n=\sum_{i=2}^nG_1Q_iF^i_n,\qquad F_n=\sum_{i=2}^nK Q_iF^i_n,\end{equation}
where 
$ F_n^i=\sum_{j_1+\cdots j_i=n}F_{j_1}\otimes\cdots\otimes F_{j_i}:\bs W^{\otimes n}\to \bs V^{\otimes i}$, are respectively the Taylor coefficients of an $A_\infty$ algebra structure on $W$ and an $A_\infty$ quasi-isomorphism $F:W\to V$.

\begin{theorem}\label{th:homotopytransfer} In the above hypotheses, if $V$ is a $C_\infty$ algebra then $R_n$, $F_n$ as in formula \eqref{eq:homotopytransfer} are the Taylor coefficients of a $C_\infty$ algebra structure on $\bs W$ and a $C_\infty$ quasi-isomorphism respectively.
\end{theorem}

This result was established by Cheng and Getzler
\cite{Cheng-Getzler} with a different proof.

\begin{proof} Suppose inductively we have shown that $R_i,F_i$ vanish on the image of the shuffle product for all $i<n$, the induction starting at $n=2$ where it is trivial: then the morphism of graded coalgebras $F_{<n}:\overline{T}(\bs W)\to\overline{T}(\bs V)$ with Taylor coefficients $(F_{<n})_i=F_i$ if $i<n$ and $(F_{<n})_i=0$ if $i\geq n$, is also a morphism of graded bialgebras. Moreover, since $V$ is a $C_\infty$ algebra the coderivation $Q_{\geq2}:\overline{T}(\bs V)\to \overline{T}(\bs V)$ with vanishing linear part $(Q_{\geq2})_1=0$ and the same higher Taylor coefficients as $Q$, $(Q_{\geq2})_i=Q_i$ if $i\geq2$, is also a biderivation. Both statements follow by the previous two lemmas. Finally, $R_n$ and $F_n$ are respectively the composition with $G_1:\bs V\to \bs W$ and $K:\bs V\to \bs V$ of the map
\[ \bs W^{\otimes n} \hookrightarrow\overline{T}(\bs W)\xrightarrow{F_{<n}} \overline{T}(\bs V)\xrightarrow{Q_{\geq2}}\overline{T}(\bs V)\xrightarrow{p} \bs V, \]
and the latter vanishes on the image of the shuffle product: in fact, so does the corestriction map $p:\overline{T}(\bs V)\to\bs V$, and by Lemma \ref{lem:F-derivation} the composition $Q_{\geq2}F_{<n}:\overline{T}(\bs W)\to\overline{T}(\bs V)$ is an $F_{<n}$-biderivation, hence it sends the image of the shuffle product in $\overline{T}(\bs W)$ into the image of the shuffle product in $\overline{T}(\bs V)$.
\end{proof}


\thebibliography{10}

\bibitem{Abad-Schaetz_integration}
C. Arias Abad and F. Sch\"atz,
{\em The $A_\infty$ de Rham theorem and integration of representations up to homotopy},
Int. Math. Res. Not. IMRN {\bf 2013}, no. 16, 3790--3855.

\bibitem{Beherend-Getzler} K. Behrend and E. Getzler, \emph{Geometric higher groupoids and categories}; \texttt{ arXiv:1508.02069v3 [math.AG]}


\bibitem{Lie models} U. Buijs, Y. F\'elix, A. Murillo and D. Tanr\'e, \emph{Lie models of simplicial sets and representability of the Quillen functor}; \texttt{arXiv: 1508.01442 [math.AT]}.


\bibitem{BMS}
F. Burghart, P. Mn\"ev and F. Steinebrunner, {\em private communication}, 2016.

\bibitem{CR}
A. S. Cattaneo and C. A. Rossi,
{\em Wilson surfaces and higher dimensional knot invariants},
Commun. Math. Phys. {\bf 256} (2005), 513--537.

\bibitem{Chen1}
K.T. Chen, 
{\em Iterated integrals of differential forms and loop space homology},
Ann. of Math. (2) {\bf 97} (1973), 217--246. 


\bibitem{Cheng-Getzler}
X.Z. Cheng and E. Getzler,
{\em Transferring homotopy commutative algebraic structures},
 J. Pure Appl. Algebra {\bf 212} (11) ( 2008), 2535--2542. 


\bibitem{Dupont2}  J. L. Dupont, \emph{Curvature and characteristic classes}, Lecture Notes in Mathematics \textbf{640}, Springer-Verlag, 1978.


\bibitem{Fiorenza-Manetti}
D. Fiorenza and M. Manetti,
{\em $L_\infty$ structures on mapping cones},
Algebra Number Theory {\bf 1} (2007), no. 3, 301--330. 

\bibitem{Getzler-Ham}
E.~Getzler, {\em A Darboux theorem for Hamiltonian operators in the formal calculus of variations}, Duke J. Math. {\bf 111} (2002), 535--560.

\bibitem{Getzler} E.~Getzler, \emph{Lie theory for nilpotent $L_{\infty}$-algebras.}, Ann. of Math. \textbf{170}, no. 1 (2009), 271--301; \texttt{arXiv:math/0404003v4}.

\bibitem{Gugenheim0}
V.K.A.M. Gugenheim, {\em On the multiplicative structure of the de Rham theory},
J. Differential Geometry {\bf 11} (1976), no. 2, 309--314. 

\bibitem{Gugenheim}
V.K.A.M. Gugenheim,
{\em On Chen's iterated integrals},
Illinois J. Math. {\bf 21} (1977), no. 3, 703--715. 

\bibitem{Adams-Hilton} K. Hess, P.-E. Parent, J. Scott and A. Tonks, \emph{A canonical enriched Adams-Hilton model for simplicial sets}, Adv. in Math. \textbf{207} (2006), no. 2, 847--875; \texttt{ arXiv: math/0408216v2 [math.AT]}.



\bibitem{hinichdgC} V. Hinich, \emph{DG coalgebras as formal stacks}, J. of Pure and Appl. Algebra \textbf{162} (2001), 209--250; \texttt{arXiv:math/9812034v1 [math.AG]}.


\bibitem{Huebschmann-Kadeishvili}
J. Huebschmann and T. Kadeishvili,
{\em Small models for chain algebras}, Math. Z. {\bf 207} (1991),
245--280.

\bibitem{Iserles-Norsett}
A. Iserles and S.P. N\o rsett, {\em On the solution of linear differential equations in Lie groups}, Royal Soc. Lond. Philos. Trans. Ser. A Math. Phys. Eng. Sci. {\bf 357} (1999), 983--1020.

\bibitem{Lawrence-Sullivan} R. Lawrence and D. Sullivan, \emph{A free differential Lie algebra for the interval}; available at \url{http://www.ma.huji.ac.il/~ruthel/papers/06bernoulli6.pdf}

\bibitem{Loday}
J.L. Loday, {\em S\'erie de Hausdorff, idempotents eul\'eriens et alg\`ebres de Hopf}, Exp. Math. {\bf 12} (1994), 165--178.

\bibitem{Magnus} W. Magnus, {\em On the exponential solution of differential equations for a linear operator},
Commun. Pure Appl. Math. \textbf{7} (1954), 649--673.

\bibitem{Majewski} M. Majewski,
\emph{Rational homotopical models a and uniqueness}, 
Mem. Amer.
Math. Soc.
\textbf{682}
(2000).


\bibitem{Markl}
M. Markl,
{\em Transferring $A_\infty$ (strongly homotopy associative) structures}, 
Rend. Circ. Mat. Palermo (2) Suppl. No. {\bf 79} (2006), 139--151. 


\bibitem{Mielnik-Plebanski} B. Mielnik and J. Pleba\'nski, \emph{Combinatorial approach to Baker-Campbell-Hausdorff exponents},
	Ann. Inst. Henri Poincare, Sect. A \textbf{XII} (1970), 215--254.

\bibitem{Miller} M. Miller, \emph{Homotopy algebra structures on twisted tensor products and string topology operations}, Alg. and Geom. Topol. \textbf{11} (2011), 1163--1203; \texttt{arXiv:1006.2781v3 [math.AT]}.
	

\bibitem{Mnev}
P. Mn\"ev, {\em Notes on simplicial BF theory},
Mosc. Math. J. {\bf 9} (2009), no. 2, 371--410.


\bibitem{Pavel-obs}
P. Mn\"ev, {\em A Construction of Observables for AKSZ Sigma Models},
Lett. Math. Phys. {\bf 105} (2015), Issue 12, 1735--1783.

\bibitem{Reutenauer}
C. Reutenauer, {\em Free Lie Algebras}, Oxford University Press, New York (1993), xviii+269 pp.
\end{document}